\documentclass{amsart}
\usepackage{amsmath,amsxtra,amssymb, color}
\usepackage[all]{xy}

\newcommand{\Span}{\mathrm{span}}
\newcommand{\conv}{\mathrm{conv}}
\newcommand{\Aut}{\mathrm{Aut}}
\newcommand{\C}{\mathbb{C}}
\newcommand{\R}{\mathbb{R}}
\newcommand{\Tr}{\mathrm{Tr}}
\newcommand{\ad}{\mathrm{ad}}
\newcommand{\id}{\mathrm{id}}
\newcommand{\cb}{\mathrm{cb}}

\newtheorem{theorem}{Theorem}[section]
\newtheorem{cor}[theorem]{Corollary}
\newtheorem{prop}[theorem]{Proposition}
\newtheorem{lemma}[theorem]{Lemma}
\newtheorem{defi}[theorem]{Definition}

\theoremstyle{definition}
\newtheorem{rem}[theorem]{Remark}

\begin{document}
\title{An asymptotic property of factorizable completely positive maps and the Connes embedding problem}
\author{Uffe Haagerup$^{(1)}$ and Magdalena Musat$^{(2)}$}
\address{$^{(1), (2)}$ Department of Mathematical Sciences, University of
Copenhagen, Universitetsparken 5, 2100 Copenhagen {\O}, Denmark.}
\email{$^{(1)}$haagerup@math.ku.dk }
\email{$^{(2)}$musat@math.ku.dk}

\date{}

\footnotetext {$^{(1)}$ Supported by the ERC Advanced Grant no. OAFPG 247321, and partially supported by the Danish National Research Foundation (DNRF) through the Centre for Symmetry and Deformation at the University of Copenhagen, and the Danish Council for Independent Research, Natural Sciences.
\\ \hspace*{0.48cm} $^{(2)}$
Supported by the Danish National Research Foundation (DNRF) through the Centre for Symmetry and Deformation at the University of Copenhagen, and the Danish Council for Independent Research, Natural Sciences.}


\maketitle

\begin{abstract}

We establish a reformulation of the Connes embedding problem in terms of an asymptotic property of factorizable completely positive maps. We also prove that the Holevo-Werner channels $W_n^-$ are factorizable, for all odd integers $n\neq 3$. Furthermore, we investigate factorizability of convex combinations of $W_3^+$ and $W_3^-$, a family of channels studied by Mendl and Wolf, and discuss asymptotic properties for these channels.

\end{abstract}

\section{Introduction}
\setcounter{equation}{0}

The class of factorizable completely positive maps (introduced by C. Anantharaman-Delaroche in \cite{AD}) has gained particular significance in quantum information theory in connection
with the settling (in the negative) of the {\em asymptotic quantum Birkhoff conjecture}. This conjecture originated in work of J. A. Smolin, F. Verstraete and A. Winter (cf. \cite{SVW}), where they provided evidence that every unital quantum channel might always be well approximated by a convex combination of unitarily implemented ones. Further support for this conjectured restoration in the asymptotic limit of Birkhoff's classical theorem was given by C. Mendl and M. Wolf in \cite{MWo}, where they presented a family of unital quantum channels outside the convex hull of the unitary ones, exhibiting the interesting property that they fall back into this set when taking the tensor product of two copies of them.

In \cite{HM6}, we proved that every non-factorizable unital completely positive and trace-preserving map on $M_n(\mathbb{C})$,
$n\geq 3$, provides a counterexample for the conjecture, and we gave examples of non-factorizable unital quantum channels in all dimensions $n\geq 3$. It was then a natural question whether every factorizable unital quantum channel does satisfy the asymptotic quantum Birkhoff property (AQBP, for short). This question turned out to have an interesting interpretation, in that it seemingly related to the celebrated Connes embedding problem (cf. \cite{Con}), known to be equivalent to a number of other fundamental problems in operator algebras. We showed in the above-mentioned paper \cite{HM6} (see Theorem 6.2 therein) that if for all $n\geq 3$, every factorizable unital quantum channel in dimension $n$ does satisfy the AQBP, then the Connes embedding problem has a positive answer. However, after the paper \cite{HM6} was submitted for publication, we discovered that the factorizable channel from Example 3.3 therein does not satisfy the AQBP, thus,
there is no direct connection between the asymptotic quantum Birkhoff property and the Connes embedding problem. We announced this result in Remark 6.3 of \cite{HM6}. Furthermore, we also announced therein that the Connes embedding problem connects, in fact, to another asymptotic property of factorizable completely positive maps. Namely, the Connes embedding problem has a positive answer if and only if the following equality holds for every $n\geq 3$ and every factorizable unital quantum channel $T$ in dimension $n$:
\begin{equation}\label{eq4444444}
    \lim\limits_{k\rightarrow \infty} d_{\text{cb}}(T\otimes {S_k}, \text{conv}(\text{Aut}(M_n(\mathbb{C})\otimes M_k(\mathbb{C})))) = 0 \,,
\end{equation}
where $S_k$ is the completely depolarizing channel on $M_k(\mathbb{C})$, i.e., $S_k(x)=\tau_k(x)1_k$, for all $x\in M_k(\mathbb{C})$. Here $\tau_k$ denotes the normalized trace on $M_k(\mathbb{C})$, and $1_k$ is the identity $k\times k$ matrix. We give the proof of these statements in Sections 2 and 3 of this paper.
We then prove that the Holevo-Werner channels $W_n^-$ are factorizable, for all odd integers $n\geq 5$, and show that they do satisfy the asymptotic property (\ref{eq4444444}) above. We have shown in \cite[Example 3.1]{HM6} that $W_3^-$ is not factorizable.
Here we investigate furthermore factorizability
of convex combinations of $W_3^+$ and $W_3^-$, a family of channels studied by Mendl and Wolf in \cite{MWo}. We also determine the cb-distance from $W_3^-$ to the factorizable maps. This is all done in Section 5. The main tool in the proof of these factorizability results is Theorem 4.5, which is the main result of Section 4. This theorem is motivated by the averaging techniques of Mendl and Wolf from \cite{MWo}, building on earlier analysis of entanglement measures under symmetry carried out by Vollbrecht and Werner in \cite{VoW}. In the last section we study further
asymptotic properties of the family $T_\lambda= \lambda W_3^{+} + (1-\lambda) W_3^{-}$, $0\leq \lambda\leq 1$. Mendl and Wolf showed in \cite{MWo} that these channels satisfy the interesting property that $T_\lambda$ belongs to the convex hull of automorphisms of $M_3(\mathbb{C})$ if and only if $\lambda\geq {1/3}$, while furthermore,
for some $0< \lambda_0< {1/3}$, one has
$T_\lambda\otimes T_\lambda \in \text{conv}(\text{Aut}(M_9(\mathbb{C})))$, for all $\lambda\in [\lambda_0, 1]$\,. Our main result in this section is Theorem 6.1 asserting that for every $\lambda\in [1/4, 1]$ and every integer $k\geq 2$\,, one has
$T_\lambda^{\otimes k}\in \text{conv}(\text{Aut}(M_{3^k}(\mathbb{C})))$\,. Hence $T_\lambda$ does satisfy the AQBP, for all $\lambda\in [1/4, 1]$.

Throughout the paper, we denote the set of unital quantum channels in dimension $n$, that is, unital completely positive trace-preserving maps on $M_n(\mathbb{C})$, by UCPT$(n)$.

\section{An example of a factorizable map which does not satisfy the asymptotic quantum Birkhoff property} \setcounter{equation}{0}

We begin this section by establishing a number of intermediate results, some of which may be of independent interest. The first one is probably known (and follows from the work of Choi \cite{Ch}), but we include a (possibly different) proof for convenience.

\begin{prop}\label{prop1}
Let $T\colon M_n(\mathbb{C})\rightarrow M_n(\mathbb{C})$ be a UCPT$(n)$-Schur multiplier.
\begin{enumerate}
\item [$(1)$] If $Tx=\sum_{i=1}^d a_i^* x a_i$\,, for all $x\in M_n(\mathbb{C})$, for some $a_1\,, \ldots \,, a_d\in M_n(\mathbb{C})$, then $a_1\,, \ldots\,, a_d$ are diagonal matrices.
\item[$(2)$] If $T\in \conv(\Aut (M_n(\mathbb{C})))$, i.e., $Tx=\sum_{i=1}^d c_i u_i^* x u_i$, $x\in M_n(\mathbb{C})$, where $c_i> 0$, $\sum_{i=1}^d c_i=1$, $u_i\in {\mathcal U}(n)$, $1\leq i\leq d$, then $u_1\,, \ldots \,, u_d$ are diagonal matrices.
\end{enumerate}
\end{prop}

\begin{proof}
$(1)$ Suppose that $Tx=\sum_{i=1}^d a_i^* x a_i$\,, $x\in M_n(\mathbb{C})$, for some $a_1\,, \ldots \,, a_d\in M_n(\mathbb{C})$. Let $p$ be a projection in $D_n(\mathbb{C})$, the set of diagonal $n\times n$ complex matrices. We then have $p=Tp=\sum_{i=1}^d a_i^* p a_i$. Therefore $\sum_{i=1}^d (p a_i (1-p))^*(p a_i (1-p))=(1-p)p(1-p)=0$, which implies that $p a_i (1-p)=0$\,, for all $1\leq i\leq d$. Similarly, $\sum_{i=1}^d ((1-p) a_i p)^* ((1-p) a_i p)=0$, which implies that $(1-p)a_i p=0$, for all $1\leq i\leq d$. Then the commutator $[a_i, p]\colon =a_i p-p a_i=(1-p) a_i p - p a_i (1-p)=0$, for all $1\leq i\leq d$. This shows that $a_i\in {{D_n(\mathbb{C})}'}\cap {M_n(\mathbb{C})}=D_n(\mathbb{C})$, $1\leq i\leq d$, as claimed.

$(2)$ follows from $(1)$, by setting $a_i=\sqrt{c_i} u_i$\,, for all $1\leq i\leq d$.
\end{proof}

\begin{theorem}\label{th121}
Let $T\colon M_n(\mathbb{C})\rightarrow M_n(\mathbb{C})$ be a UCPT$(n)$-Schur multiplier, and $S\colon M_k(\mathbb{C})\rightarrow M_k(\mathbb{C})$ a UCPT$(k)$-Schur multiplier, where $n, k$ are positive integers. The following statements are equivalent:
\begin{enumerate}
\item [$(1)$] $T\otimes S\in \conv(\Aut (M_n(\mathbb{C})\otimes M_k(\mathbb{C})))$.
\item [$(2)$] $T\in \conv(\Aut(M_n(\mathbb{C})))$ and $S\in  \conv(\Aut(M_k(\mathbb{C})))$.
\end{enumerate}
\end{theorem}

\begin{proof} The implication $(2)\Rightarrow (1)$ is clear, so we proceed to showing that $(1)\Rightarrow (2)$. Let $(e_{ij})_{1\leq i, j\leq n}$ and $(f_{st})_{1\leq s, t\leq k}$ be the canonical matrix units in $M_n(\mathbb{C})$ and $M_k(\mathbb{C})$, respectively.
If $T\otimes S\in \text{conv}(\text{Aut}(M_n(\mathbb{C})\otimes M_k(\mathbb{C})))$, then by Proposition \ref{prop1}, there exist a positive integer $m$ and numbers $c_i> 0$, $1\leq i\leq m$ with $\sum_{i=1}^m c_i=1$, as well as diagonal unitaries $u_1, \ldots , u_m\in D_{nk}(\mathbb{C})=D_n(\mathbb{C})\otimes D_k(\mathbb{C})$, such that
\[ (T\otimes S)(y)=\sum\limits_{i=1}^m c_i u_i^* y u_i\,, \quad y\in M_n(\mathbb{C})\otimes M_k(\mathbb{C}). \]
For all $1\leq i\leq m$, $v_i\colon =u_i({1_n}\otimes {f_{11}})$ is a unitary in $(1_n\otimes f_{11})(M_n(\mathbb{C})\otimes M_k(\mathbb{C}))(1_n\otimes f_{11})\simeq M_n(\mathbb{C})\otimes f_{11}.$ Hence, there exist unitaries $w_i\in M_n(\mathbb{C})$ such that
$v_i={w_i}\otimes f_{11}$\,, $1\leq i\leq m$.
Then, for all $x\in M_n(\mathbb{C})$,
\begin{eqnarray}\label{eq15}
(T\otimes S)(x\otimes f_{11})&=& (1_n\otimes f_{11})((T\otimes S)(x\otimes f_{11}))(1_n\otimes f_{11})\\
&=&\left(\sum\limits_{i=1}^m c_i w_i^* x w_i\right) \otimes f_{11}.\nonumber
\end{eqnarray}
Since $S(f_{11})=f_{11}$, we infer from (\ref{eq15}) that
$Tx=\sum\limits_{i=1}^m c_i w_i^* x w_i$\,, for all $x\in M_n(\mathbb{C})$,
which implies that $T\in \text{conv}(\text{Aut}(M_n(\mathbb{C})))$. A similar proof shows that $S\in \text{conv}(\text{Aut}(M_n(\mathbb{C})))$.
\end{proof}

\begin{rem}\label{rem11} The Schur multiplier $T_B$ constructed in Example 3.3 of \cite{HM6} is a factorizable UCPT$(6)$-map with the property that $T_B\notin \text{conv}(\text{Aut}(M_6(\mathbb{C})))$. In view of the above theorem, it now follows that $T_B^{\otimes n} \notin  \text{conv}(\text{Aut}(M_{6^n}(\mathbb{C})))$\,, for any $n\geq 2$.

The next result shows that $T_B$ does not satisfy the asymptotic quantum Birkhoff property.
\end{rem}

\begin{theorem}\label{th122}
Let $T$ be a UCPT$(n)$-Schur multiplier and $S$ a UCPT$(k)$-Schur multiplier, where $n, k$ are positive integers. Then
\begin{eqnarray*}
d_{\text{cb}}(T\otimes S, \conv(\Aut(M_{nk}(\mathbb{C}))))\geq \frac{1}{2} d_{\text{cb}}(T, \conv(\Aut(M_n(\mathbb{C}))))\,.
\end{eqnarray*}
In particular, if $T\notin  \conv(\Aut(M_n(\mathbb{C}))$, then $T$ fails the asymptotic quantum Birkhoff property.
\end{theorem}

\begin{proof}
Let $\alpha=d_{\text{cb}}(T\otimes S, \text{conv}(\text{Aut}(M_{nk}(\mathbb{C}))))$. Then there exists $m\in \mathbb{N}$, and for $1\leq j\leq m$, there exist $c_j> 0$ with $\sum_{j=1}^m c_j=1$ and unitary ${nk}\times {nk}$ matrices $u_j$ such that
$\left\|T\otimes S-\sum_{j=1}^m c_j u_j^* x u_j\right\|=\alpha$\,.

As before, let $(f_{st})_{1\leq s, t\leq k}$ be the canonical matrix units in $M_k(\mathbb{C})$. Then, for every $1\leq j\leq m$, there exists $b_j\in M_n(\mathbb{C})$ such that
$(1_n\otimes f_{11})u_j(1_n\otimes f_{11})=b_j\otimes f_{11}$\,.
Next, set $R(x)=\sum_{j=1}^m c_j b_j^* x b_j$\,, $x\in M_n(\mathbb{C})$\,. Then $R$ is a completely positive map, and we claim that
\begin{equation}\label{eq2323}
\|T-R\|_{\text{cb}}\leq  \alpha\,.
\end{equation}
To prove this, note first that for all $z\in M_n(\mathbb{C})\otimes M_k(\mathbb{C})$, $\|(T\otimes S)(z)-\sum_{j=1}^m c_j u_j^* z u_j\|\leq \alpha \|z\|$\,.
In particular, using that $S(f_{11})=f_{11}$, it follows for all $x\in M_n(\mathbb{C})$ that
\[ \left\|T(x)\otimes f_{11}-\sum_{j=1}^m c_j u_j^* (x\otimes f_{11}) u_j\right\|\leq \alpha \|x\|\,. \]
This implies that $\|(1_n\otimes f_{11})(T(x)\otimes f_{11}-\sum_{j=1}^m c_j u_j^* (x\otimes f_{11}) u_j)(1_n\otimes f_{11})\|\leq \alpha\|x\|$\,. Equivalently,
\[ \left\|T(x)\otimes f_{11}-\left(\sum_{j=1}^m c_j b_j^* x b_j\right)\otimes f_{11}\right\|\leq \alpha \|x\|\,, \quad x\in M_n(\mathbb{C})\,, \]
which shows that $\|T-S\|\leq \alpha$\,. A similar argument applied to $T\otimes \text{id}_{l}$\,, $l\geq 2$ yields (\ref{eq2323}).

Next, for $1\leq j\leq m$, since $\|b_j\|\leq 1$, we have $b_j={(v_j+w_j)}/2$, for some $n\times n$ unitaries $v_j$\,, $w_j$. Further, set $\tilde{T}(x)=(1/2)\sum_{j=1}^m c_j(v_j^* x v_j+w_j^* x w_j)$\,, $x\in M_n(\mathbb{C})$\,. Then $\tilde{T}\in \text{conv}(\text{Aut}(M_n(\mathbb{C}))$. We claim that
\begin{equation}\label{eq2324}
\|\tilde{T} -R\|\leq \alpha\,.
\end{equation}
Note that $(\tilde{T}-R)(x)=({1}/{4}) \sum_{j=1}^m c_j (v_j-w_j)^* x (v_j-w_j)$\,, $x\in M_n(\mathbb{C})$,
hence $T^\prime-R$ is completely positive. Therefore
\begin{equation}\label{eq45465}
\|\tilde{T}-R\|_{\text{cb}} = \|(\tilde{T}-R)(1_n)\|=\frac{1}{4} \left\|\sum_{j=1}^m c_j (v_j-w_j)^*(v_j-w_j)\right\|=\left\|\sum_{j=1}^m c_j(1_n-b_j^*b_j)\right\|\,.
\end{equation}
By using (\ref{eq2323}),
\[ \left\|\sum_{j=1}^m c_j (1_n-b_j^* b_j)\right\|=\|1_n-R(1_n)\|=\|T(1_n)-R(1_n)\|=\|T-R\|_{\text{cb}}\leq \alpha\,.\]
Combined with (\ref{eq45465}), this yields (\ref{eq2324}). An application of the triangle inequality gives that
\[ \|T-\tilde{T}\|_{\text{cb}}\leq 2\alpha\,. \]
Since $\tilde{T}\in \text{conv}(\text{Aut}(M_n(\mathbb{C})))$, the conclusion follows.
\end{proof}

\section{Tensoring with the completely depolarizing channel and a new asymptotic property} \setcounter{equation}{0}

\begin{defi}\label{defexactfactoriz}
Let $T\colon M_n(\mathbb{C})\rightarrow M_n(\mathbb{C})$ be a UCPT$(n)$-map. We say that $T$ has an {\bf exact factorization} through $M_n(\mathbb{C})\otimes N$, for some von Neumann algebra $N$ with a normal, faithful, tracial state $\tau_N$, if there exists a unitary $u\in M_n(\mathbb{C})\otimes N$ such that
\begin{equation}\label{eq5657}
T(x)=(\id_n\otimes \tau_N)(u^*(x\otimes 1_N)u)\,, \quad x\in M_n(\mathbb{C})\,.
\end{equation}
\end{defi}

\begin{rem}\label{rem789}
By (the proof of) Theorem 2.2 in \cite{HM6}, a UCPT$(n)$-map $T$ has an exact factorization through $M_n(\mathbb{C})\otimes N$, for some von Neumann algebra $N$ with a normal, faithful, tracial state $\tau_N$ if and only if $T$ is factorizable in the following more precise sense, that there exist unital completely positive $(\tau_n, \tau_n\otimes \tau_N)$-preserving maps $\alpha, \beta\colon M_n(\mathbb{C})\rightarrow M_n(\mathbb{C})\otimes N$ such that $T=\beta^*\circ \alpha$.
\end{rem}

We now introduce another definition:

\begin{defi}\label{deffactorizdeg}
A UCPT$(n)$-map $T\colon M_n(\mathbb{C})\rightarrow M_n(\mathbb{C})$ is said to be {\bf factorizable of degree} $k$, for some integer $k\geq 1$, if
\begin{equation}\label{eq878}
T\otimes S_k\in \conv\big(\Aut (M_n(\mathbb{C})\otimes M_k(\mathbb{C}))\big),
\end{equation}
where $S_k$ is the {\bf completely depolarizing} channel on $M_k(\mathbb{C})$, i.e., $S_k(y)=\tau_k(y) \, 1_k$\,, for all $y\in M_k(\mathbb{C})$.
\end{defi}

The following result establishes the connection between these notions.

\begin{prop}\label{prop348}
Let $T$ be a UCPT$(n)$-map. Then $T$ is factorizable of degree $k$, for some integer $k\geq 1$, if and only if $T$ has an exact factorization through $M_n(\mathbb{C})\otimes N$, where $N=M_k(\mathbb{C})\otimes L^{\infty}([0, 1], m)$\,. (Here $m$ denotes the Lebesgue measure on $[0, 1]$.)
\end{prop}

\begin{proof}
Suppose that $T$ is factorizable of degree $k$, for some integer $k\geq 1$, i.e., \[ T\otimes S_k\in \text{conv}(\text{Aut}(M_n(\mathbb{C})\otimes M_k(\mathbb{C}))). \] The following arguments are in the spirit of the proof of Proposition 2.4 in \cite{HM6}. Write $T\otimes S_k=\sum_{j=1}^r c_j \text{ad} (u_j)$, for some $r\in \mathbb{N}$, $c_j>0$\,, $1\leq j\leq r$, with $\sum_{j=1}^r c_j =1$, and unitaries $u_j\in M_n(\mathbb{C})\otimes M_k(\mathbb{C})$, $1\leq j\leq r$.
Then there exist projections $p_1\,, \ldots\,, p_r\in L^\infty ([0, 1], m)$, where $m$ is the Lebesgue measure on $[0, 1]$, such that $1_{L^\infty([0, 1])}=p_1+\ldots + p_r$ and $\tau(p_j)=c_j$, $1\leq j\leq r$, where
$\tau(f)=\int_{[0,1]} f \,dm$, for all $f\in L^\infty([0, 1])$.
Further, let $N=M_k(\mathbb{C})\otimes L^\infty([0, 1], m)$, with trace $\tau_N=\tau_k\otimes \tau$, and set \[ u=\sum_{j=1}^r u_j\otimes p_j\in M_n(\mathbb{C})\otimes N. \]
Note that $u$ is unitary, and for all $y\in M_n(\mathbb{C})\otimes M_k(\mathbb{C})$,
$u^*(y\otimes  {1_{L^\infty([0, 1])}})u=\sum_{j=1}^r u^*_j y u_j \otimes p_j$. Thus,
\begin{equation}\label{eq999}(T\otimes S_k)(y)=\sum_{j=1}^r c_j u_j^* y u_j=(\text{id}_{n}\otimes \text{id}_{k}\otimes \tau)(u^*(y\otimes 1_{L^\infty([0, 1])})u)\,, \quad y\in M_n(\mathbb{C})\otimes M_k(\mathbb{C})\,. \end{equation}
Further, note that $T=(\text{id}_{n}\otimes \tau_k)\circ (T\otimes S_k)$. Combining this with $(\ref{eq999})$, we deduce that
\begin{equation*}
T(x)=(\text{id}_{n}\otimes \tau_N)(u^*(x\otimes 1_N)u)\,, \quad x\in M_n(\mathbb{C})\,.
\end{equation*}
That is, $T$ has an exact factorization through $M_n(\mathbb{C})\otimes N$.

Conversely, assume that $T$ has an exact factorization through $M_n(\mathbb{C})\otimes N$\,, where $N=M_k(\mathbb{C})\otimes L^{\infty}([0, 1], m)$, with $m$ being the Lebesgue measure on $[0, 1]$. Using the identification $N=L^\infty([0, 1], M_k(\mathbb{C}))$, the trace $\tau_N$ on $N$ is given by \begin{equation}\label{eq66667} \tau_N(y)=\int_{[0, 1]} \tau_k(y(t)) dm(t), \quad y\in N\,. \end{equation}
By the hypothesis, there exists a unitary $u$ in $M_n(N)=L^\infty([0, 1], M_n(\mathbb{C})\otimes M_k(\mathbb{C}))$ such that
\[ T(x)=(\text{id}_{n}\otimes \tau_N)(u^*(x\otimes 1_N)u)\,, \quad x\in M_n(\mathbb{C})\,. \]
Under the above identification, $u(t)$ is a unitary in $M_n(\mathbb{C})\otimes M_k(\mathbb{C})$, for all $t\in [0, 1]$\,.

We claim that for all $x\in M_n(\mathbb{C})$ and $y\in M_k(\mathbb{C})$,
\begin{equation}\label{eq7787}
(T\otimes S_k)(x\otimes y)=\int_{[0,1]} \int_{{\mathcal U}(k)} \int_{{\mathcal U}(k)}(1_n\otimes w)^*{u(t)}^*(1_n\otimes v)^*(x\otimes y)(1_n\otimes v)u(t)(1_n\otimes w) dv dw dm(t)\,,
\end{equation}
which, by interpreting the iterated integrals as a limit of convergent Riemann sums, yields the conclusion.
The proof of (\ref{eq7787}) will be achieved through a few intermediate steps. First, since
\[ \int_{{\mathcal U}(k)} v^* y v dv=\tau_k(y)1_k=S_k(y)\,, \quad y\in M_k(\mathbb{C})\,, \]
we can rewrite the right-hand side of (\ref{eq7787}) as
\begin{eqnarray}\label{eq8898}
&&\hspace*{-1.4cm}\int_{[0,1]} \int_{{\mathcal U}(k)} \int_{{\mathcal U}(k)}(1_n\otimes w)^*{u(t)}^*(1_n\otimes v)^*(x\otimes y)(1_n\otimes v)u(t)(1_n\otimes w) dv \,dw \,dm(t)\\
&=&\int_{[0, 1]} \int_{{\mathcal U}(k)} (1_n\otimes w)^*{u(t)}^*\left(x\otimes \int_{{\mathcal U}(k)} v^* y v dv\right)u(t) (1_n\otimes w) dw \,dm(t)\nonumber\\
&=& \tau_k(y) \int_{{\mathcal U}(k)} (1_n\otimes w)^*{u(t)}^*(x\otimes 1_k) u(t) (1_n\otimes w) dw \,dm(t)\,.\nonumber
\end{eqnarray}
Next, observe that for all $z\in M_n(\mathbb{C})\otimes M_k(\mathbb{C})$,
\begin{equation}\label{eq909}
\int_{{\mathcal U}(k)} (1_n\otimes w)^* z (1_n\otimes w) dw=(\text{id}_{n}\otimes S_k)(z)=(\text{id}_{n}\otimes \tau_k)(z)\otimes 1_k\,,
\end{equation}
where both equalities can be checked easily on elementary tensors $z=a\otimes b$\,, where $a\in M_n(\mathbb{C})$, $b\in M_k(\mathbb{C})$. In particular, by using
(\ref{eq909}) with  $z={u(t)}^*(x\otimes 1_n)u(t)\in M_n(\mathbb{C})\otimes M_k(\mathbb{C})$, $t\in [0, 1]$, we get
\begin{eqnarray}
&&\hspace*{-1.4cm}\tau_k(y)\int_{[0, 1]} \int_{{\mathcal U}(k)} (1_n\otimes w)^*{u(t)}^*(x\otimes 1_k) u(t)(1_n\otimes w) dw dm(t)\label{eq66666}\\ &=& \tau_k(y) \int_{[0,1]} {(\text{id}_{n}\otimes \tau_k)({u(t)}^*(x\otimes 1_k)u(t))\otimes 1_k} \,\,dm(t)\nonumber\\&=&(\text{id}_n\otimes \tau_N)(u^* (x\otimes 1_N) u) \otimes S_k(y)\nonumber\\
&=&T(x)\otimes S_k(y)\,,\nonumber
\end{eqnarray}
wherein we have used $(\ref{eq66667})$ and the fact that under the identification $N=L^\infty([0, 1], M_k(\mathbb{C}))$, the identity $1_N$ of $N$ is given by $1_N(t)=1_k$, for all $t\in [0, 1]$. Combining $(\ref{eq66666})$ with (\ref{eq8898}) gives the conclusion.
\end{proof}

\begin{cor}
If a UCPT$(n)$-map $T$ has an exact factorization through $M_n(\mathbb{C})\otimes M_k(\mathbb{C})$, for some $k\geq 1$, then $T$ is factorizable of degree $k$.
\end{cor}

In the following we give a characterization of the UCPT$(n)$-maps which admit an exact factorization through a von Neumann algebra embeddable into an ultrapower ${\mathcal R}^{\omega}$ of the hyperfinite II$_1$-factor ${\mathcal R}$.

\begin{theorem}\label{th11}
Let $T$ be a factorizable UCPT$(n)$-map. The following statements are equivalent:
\begin{itemize}
\item [$(1)$] $T$ has an exact factorization through $M_n(\mathbb{C})\otimes N$, where $N$ is a finite von Neumann algebra which embeds into ${\mathcal R}^{\omega}$.
\item [$(2)$] There exists a sequence $(T_k)_{k\geq 1}$ of UCPT$(n)$-maps, where
each $T_k$ has an exact factorization through $M_n(\mathbb{C})\otimes M_{l(k)}(\mathbb{C})$, for some integer $l(k)\geq 1$, such that  $\|T-T_k\|_{\text{cb}}\rightarrow 0$\,, as $k\rightarrow \infty$\,.
\item[$(3)$] $\lim\limits_{k\rightarrow \infty} d_{\text{cb}}(T\otimes S_k, \conv(\Aut (M_n(\mathbb{C})\otimes M_k(\mathbb{C}))))=0$\,.
\end{itemize}
\end{theorem}

\begin{proof}
The proof of $(1)\Rightarrow (2)$ is based on standard ultraproduct arguments and uses also some of the ideas of the proof of Theorem 6.2 in \cite{HM6}.  For the sake of completeness of exposition, we include the details. Let us begin by recalling the necessary background. Given a free ultrafilter $\omega$ on $\mathbb{N}$, the ultrapower ${\mathcal R}^\omega$ of the hyperfinite II$_1$-factor ${\mathcal R}$ is the quotient space ${\mathcal R}^\omega={{\ell^\infty}({\mathcal R})}/I$, where $I=\{(x_k)_{k\geq 1}\in {\ell^\infty}({\mathcal R}): \lim_{\omega} \|x_k\|_2=0\}$. Let $\pi\colon {\ell^\infty}({\mathcal R})\rightarrow {\mathcal R}^\omega$ denote the quotient map. Then ${\mathcal R}^\omega$ is a II$_1$-factor with unique trace $\tau_{{\mathcal R}^\omega}$ satisfying
\begin{equation}\label{eq33333}
\tau_{{\mathcal R}^\omega}(\pi(x))=\lim_\omega \tau_{\mathcal R}(x_k)\,, \quad x=(x_k)_{k\geq 1}\in {\ell^\infty}({\mathcal R})\,.
\end{equation}
Consider the map $\text{id}_n\otimes \pi\colon M_n(\mathbb{C})\otimes {\ell^\infty}({\mathcal R})\rightarrow M_n(\mathbb{C})\otimes {\mathcal R}^\omega$\,. We identify $M_n(\mathbb{C})\otimes \ell^\infty({\mathcal R})=\ell^\infty(M_n(\mathbb{C})\otimes {\mathcal R})$\,.
Let $y\in M_n(\mathbb{C})\otimes {\mathcal R}^\omega$ and find $x=(x_k)_{k\geq 1}\in \ell^\infty(M_n(\mathbb{C})\otimes {\mathcal R})$ such that $(\text{id}_n\otimes \pi)(x)=y$\,. By $(\ref{eq33333})$\,,
\begin{equation}\label{eq333333}
(\text{id}_n\otimes \tau_{{\mathcal R}^\omega})(y)=\lim_\omega (\text{id}_n\otimes \tau_{\mathcal R})(x_k)\,.
\end{equation}
The convergence in $(\ref{eq333333})$ is a priori entry-wise convergence in $M_n(\mathbb{C})$. However, since all vector space topologies on the finite dimensional space $M_n(\mathbb{C})$ are the same, we conclude that $(\ref{eq333333})$ holds with convergence with respect to the operator norm on $M_n(\mathbb{C})$.

By hypothesis, there exists a von Neumann algebra $N$ with a normal faithful tracial state $\tau_N$ such that $N$ embeds into ${\mathcal R}^{\omega}$, as well as a unitary $u\in M_n(\mathbb{C})\otimes N$ so that
$Tx=(\text{id}_{n}\otimes \tau_N)(u^*(x\otimes 1_N)u)$\,, for all $x\in M_n(\mathbb{C})$\,.
Since we can view $u$ as a unitary in $M_n(\mathbb{C})\otimes {\mathcal R}^{\omega}$, the above relation can be rewritten as
\begin{equation*}
Tx=(\text{id}_{n}\otimes \tau_{{\mathcal R}^{\omega}})(u^*(x\otimes 1_{{\mathcal R}^{\omega}})u)\,, \quad x\in M_n(\mathbb{C})\,.
\end{equation*}
The goal is to show that for every $\varepsilon> 0$, there exists a UCPT$(n)$-map $T_0$ such that $\|T-T_0\|_{\text{cb}}< \varepsilon$, and $T_0$ has an exact factorization through $M_n(\mathbb{C})\otimes M_l(\mathbb{C})$, for some integer $l\geq 1$.

Let $v=(v_k)_{k\geq 1}\in \ell^\infty(M_n(\mathbb{C})\otimes {\mathcal R})$ be a unitary lift of $u$, i.e., $(\text{id}_n\otimes \pi)(v)=u$, and each $v_k\in M_n(\mathbb{C})\otimes {\mathcal R}$ is unitary.
For every $k\geq 1$, define $V_k\colon M_n(\mathbb{C})\rightarrow M_n(\mathbb{C})$ by \[ V_k(x)=(\text{id}_{n}\otimes \tau_{\mathcal R})(v_k^*(x\otimes 1_{\mathcal R})v_k)\,, \quad x\in M_n(\mathbb{C}). \]
Since $(v_k^*(x\otimes 1_{\mathcal R})v_k)_{k\geq 1}$ is a lift of $u^*(x\otimes 1_{\mathcal R})u)$, it follows from $(\ref{eq333333})$ that for all $x\in M_n(\mathbb{C})$,
\begin{eqnarray}
T(x)&=& (\text{id}_n\otimes \tau_{{\mathcal R}^\omega})(u^*(x\otimes 1_{{\mathcal R}^\omega})u)\label{eq222222}\\
&=&\lim_\omega (\text{id}_n\otimes \tau_{{\mathcal R}})(v_k^*(x\otimes 1_{\mathcal R})v_k)\,\,=\,\,\lim_\omega V_k(x)\nonumber\,,
\end{eqnarray}
that is, $\lim_\omega V_k=T$, where the convergence is with respect to the point-norm topology. Since the space of linear maps from $M_n(\mathbb{C})$ to $M_n(\mathbb{C})$ is finite dimensional, this implies that $(V_k)_{k\geq 1}$ converges to $T$ in cb-norm, as well. Hence, given $\varepsilon> 0$,
there exists $k_0\in \mathbb{N}$ such that
\begin{equation}\label{eq67544}
\|T-V_{k_0}\|_{\text{cb}}< {\varepsilon}/2\,.
\end{equation}

Further, since ${\mathcal R}={\overline{\cup_j A_j}}^{\,\,\text{s.o.t}}$, where $A_1\subseteq A_2\subseteq \ldots $ are unital finite dimensional factors, $A_j\simeq M_{2^j}(\mathbb{C})$, it follows from Corollary 5.3.7 in Vol.\ I of \cite{KR} that there is a sequence $(w_j)_{j\geq 1}$ of unitaries, $w_j\in M_n(\mathbb{C})\otimes A_j$, converging in strong operator topology to the unitary $v_{k_0}\in M_n(\mathbb{C})\otimes {\mathcal R}$.
For every $j\geq 1$, define
\[ T_j(x)=(\text{id}_{n}\otimes \tau_{A_j})(w_j^*(x\otimes 1_{A_j})w_j)\,, \quad x\in M_n(\mathbb{C})\,, \]
where $\tau_j$ is the normalized trace on $A_j$. By construction, $T_j$ has an exact factorization through $M_n(\mathbb{C})\otimes A_j$. As above, we can view $w_j\in M_n(\mathbb{C})\otimes A_j$ as a unitary in $M_n(\mathbb{C})\otimes {\mathcal R}$, and therefore rewrite
\[ T_j(x)=(\text{id}_{n}\otimes \tau_{{\mathcal R}})(w_j^*(x\otimes 1_{\mathcal R})w_j)\,, \quad x\in M_n(\mathbb{C})\,. \]
Since the sequence $(w_j^*(x\otimes 1_{\mathcal R})w_j)_{j\geq 1}$ converges in weak operator topology to $v_{k_0}^*(x\otimes 1_{\mathcal R})v_{k_0}$, and $\text{id}_n\otimes \tau_{\mathcal R}$ is w.o.t.-continuous, we deduce that the sequence $T_j$ converges to $V_{k_0}$, a priori point-entry-wise, hence as argued above, in cb-norm. Therefore, there is some $j_0\geq 1$ such that
\begin{equation}\label{eq555555}
\|T_{j_0}-V_{k_0}\|_{\text{cb}}< {\varepsilon}/2\,.
\end{equation}
Combining this with $(\ref{eq67544})$, we deduce that $\|T-T_{j_0}\|_{\text{cb}}< \varepsilon$, as wanted.

We now prove that $(2)\Rightarrow (3)$. For every $k\in \mathbb{N}$, set $\delta_k=\inf\{\|T-T^\prime\|_{\text{cb}}\}$\,,
where the infimum is taken over all UCPT$(n)$-maps $T^\prime$  having an exact factorization through $M_n(\mathbb{C})\otimes M_k(\mathbb{C})$\,.
Note that this infimum is actually attained, so it can be replaced by minimum. Further, observe that condition $(2)$ shows that
$\inf_{k\in \mathbb{N}} \delta_k =0$\,.
In the following we will show that this actually implies that
\begin{equation}\label{eq6667}
\lim_{k\rightarrow \infty} \delta_k=0\,,
\end{equation}
which yields $(3)$.
To prove $(\ref{eq6667})$, we first claim that for every $k, l\in \mathbb{N}$, we have
\begin{equation}\label{eq7770}
\delta_{k+l}\leq \frac{k}{k+l} {\delta_k} + \frac{l}{k+l} {\delta_l}\,.
\end{equation}
Indeed, given $k, l\in \mathbb{N}$, we can find unitaries $u_k$ in $M_n\otimes M_k$ and $u_l$ in $M_n\otimes M_l$
such that the maps defined by
$U_k(x)=(\text{id}_{n}\otimes \tau_k)(u_k(x\otimes 1_n)u_k^*)$ and $U_l(x)=(\text{id}_{n}\otimes \tau_l)(u_l(x\otimes 1_n)u_l^*)$\,, $x\in M_n(\mathbb{C})$
satisfy $\|T-U_k\|_{\text{cb}}=\delta_k$\,, respectively, $\|T-U_l\|_{\text{cb}}=\delta_l$\,. Set
\[ U(x)=(\text{id}_{n}\otimes \tau_{k+l})((u_k\oplus u_l)(x\otimes 1_{k+l})(u_k^*\oplus u_l^*))\,, \quad x\in M_n(\mathbb{C})\,. \]
It is easily checked that
\begin{equation}\label{eq778780}
U(x)=\frac{k}{k+l} U_k(x) + \frac{l}{k+l} U_l(x)\,, \quad x\in M_n(\mathbb{C})\,,
\end{equation}
from which $(\ref{eq7770})$ follows.

We are now ready to prove $(\ref{eq6667})$. Let $\varepsilon > 0$ and find $j\in \mathbb{N}$ such that $\delta_j< {\varepsilon}/{2}$. Then, by $(\ref{eq7770})$, $\delta_{kj}< {\varepsilon}/{2}$, for all $k\in \mathbb{N}$. Set $C= \max\{\delta_1\,, \ldots\,, \delta_{j-1}\}$ and choose $k_0\in \mathbb{N}$ such that ${C/{k_0}}<{\varepsilon/2}$. Set $N= {k_0}j$\,. Let $m\geq N$. Then $m=kj+l$\,, for some $k\geq k_0$ and $0\leq l\leq j-1$\,. By  $(\ref{eq7770})$,
\[ \delta_m\leq \frac{kj}{m} \delta_{kj} + \frac{l}{m} \delta_l\leq \frac{kj}{m}\cdot {\frac{\varepsilon}{2}} + \frac{l}{m} C< \varepsilon\,, \]
which gives $(\ref{eq6667})$, and completes the proof of $(3)$.

Finally, we show that $(3)\Rightarrow (1)$. Suppose that $T$ satisfies condition $(3)$. There exists a sequence $(\varepsilon_k)_{k\geq 1}$ of positive numbers converging to zero so that for every $k\geq 1$, there is $V_k\in \text{conv}(\text{Aut}(M_n(\mathbb{C})\otimes M_k(\mathbb{C})))$ satisfying
\begin{equation}\label{eq44321}
\|T\otimes S_k-V_k\|_{\text{cb}}< {\varepsilon_k}.
\end{equation}
Set $T_k(x)=((\text{id}_{n}\otimes \tau_k)\circ V_k)(x\otimes 1_k)$, $x\in M_n(\mathbb{C})$\,.
By the proof of Proposition \ref{prop348}, we conclude that $T_k$ has an exact factorization through $M_n(\mathbb{C})\otimes N_k$, where $N_k=M_k(\mathbb{C})\otimes L^\infty([0, 1])$. Note that $N_k$ embeds into ${\mathcal R}$, hence there exists a unitary $u_k$ in $M_n(\mathbb{C})\otimes {\mathcal R}$ such that
\[ T_k(x)=(\text{id}_{n}\otimes \tau_{\mathcal R})(u_k^*(x\otimes 1_{\mathcal R})u_k)\,, \quad x\in M_n(\mathbb{C}). \]
Since $T-T_k=(\text{id}_{n}\otimes \tau_k)\circ (T\otimes S_k-V_k)$, it follows by $(\ref{eq44321})$ that
\begin{equation*}\label{eq22222}
\lim\limits_{k\rightarrow \infty} \|T-T_k\|_{\text{cb}}=0\,.
\end{equation*}
Let $u\colon=(\text{id}_{n}\otimes \pi)((u_k)_{k\geq 1})\in M_n(\mathbb{C})\otimes {\mathcal R}^\omega$, where, as before, $\pi\colon \ell^\infty({\mathcal R})\rightarrow {\mathcal R}^\omega$ is the canonical quotient map. Then $u$ is a unitary in $M_n(\mathbb{C})\otimes {\mathcal R}^\omega$, and, moreover,
\[ T(x)=(\text{id}_{n}\otimes \tau_{{\mathcal R}^\omega})(u^*(x\otimes 1_{{\mathcal R}^\omega})u)\,, \quad x\in M_n(\mathbb{C})\,, \]
which proves $(1)$\,.
\end{proof}

Based on this, we now establish the following reformulation of the Connes embedding problem (cf. \cite{Con}) whether every finite von Neumann algebra (on a separable Hilbert space) embeds into ${\mathcal R}^{\omega}$\,.

\begin{theorem}\label{th22323232}
The Connes embedding problem has a positive answer if and only if every factorizable UCPT$(n)$-map satisfies one of the three equivalent conditions in Theorem \ref{th11}, for all $n\geq 3$.
\end{theorem}

\begin{proof}
If the Connes embedding problem has a positive answer, then clearly every factorizable UCPT$(n)$-map satisfies condition $(1)$ in Theorem \ref{th11}, for every integer $n\geq 3$\,.

Conversely, suppose that every factorizable UCPT$(n)$-map satisfies one of the three equivalent conditions in Theorem \ref{th11}, for all $n\geq 3$. Assume by contradiction that the Connes embedding problem has a negative answer. Then, as shown by Dykema and Juschenko \cite{DyJ}, based on a refinement of Kirchberg's deep results from \cite{Ki}, there exists a positive integer $n$ such that ${{\mathcal G}_n\setminus {{\mathcal F}_n}}\neq \emptyset$\,. Recall that
${\mathcal F}_n$ is defined in \cite{DyJ} as the closure of the union over $k\geq 1$ of sets of $n\times n$ complex matrices $(b_{ij})_{1\leq i, j\leq n}$ such that  $b_{ij}=\tau_k(u_iu_j^*)$\,, where $u_1, \ldots, u_n$ are unitary $k\times k$ matrices, while
${\mathcal G}_n$ consists of all $n\times n$ complex matrices $(b_{ij})_{1\leq i, j\leq n}$ such that $b_{ij}=\tau_M(u_iu_j^*)$\,, where $u_1, \ldots, u_n$ are unitaries in some von Neumann algebra $M$ equipped with normal faithful tracial state $\tau_M$ (where $M$ varies).
 Let $B\in {{\mathcal G}_n\setminus {{\mathcal F}_n}}$\,. By \cite[Proposition 2.8]{HM6}, it follows that the associated Schur multiplier $T_B$ is factorizable. By the hypothesis, $T_B$ has an exact factorization through ${\mathcal R}^\omega$\,. Then, as shown in proof of Theorem 6.2 in \cite{HM6}, this implies that $B\in {{\mathcal F}_n}$, thus yielding a contradiction.
\end{proof}

\section{The Mendl--Wolf/Vollbrecht-Werner averaging method} \label{sec:M-W}
\setcounter{equation}{0}

\noindent The main result of this section (Theorem \ref{thm:HW} below) is motivated by the averaging techniques used by Mendl and Wolf in \cite{MWo}, building on the analysis of entanglement measures under symmetry done by Vollbrecht and Werner in \cite{VoW} (see also the further references given therein).

Let $H$ be an $n$-dimensional Hilbert space. Split  the tensor product $H \otimes H$ into its symmetric and anti-symmetric parts:
$$(H \otimes H)^+ = \Span\{ \xi \otimes \eta + \xi \otimes \eta : \xi, \eta \in H\}, \qquad
(H \otimes H)^- = \Span\{ \xi \otimes \eta - \xi \otimes \eta : \xi, \eta \in H\},$$
and note that
$$\dim (H \otimes H)^+ = {n(n+1)}/{2}, \qquad \dim (H \otimes H)^- = {n(n-1)}/{2}.$$
With $(e_{ij})_{1\leq i, j\leq n}$ being the canonical matrix units for $M_n(\C)$, consider the so-called {\em flip symmetry}
\begin{equation} \label{eq:S}
s_n = \sum_{i,j=1}^n e_{ij} \otimes e_{ji} \in M_n(\C) \otimes M_n(\C),
\end{equation}
which interchanges the factors in the tensor products $\C^n \otimes \C^n$ and  $M_n(\C) \otimes M_n(\C)$, i.e., $s_n(\xi \otimes \eta) = \eta \otimes \xi$ for $\xi, \eta \in \C^n$, and $s_n(a\otimes b)s_n^* = b \otimes a$ for $a,b \in M_n(\C)$. The spectral projections
\begin{equation} \label{eq:PS}
p_n^+ = \textstyle{\frac12} (1_n+s_n), \qquad p_n^- = \textstyle{\frac12} (1_n-s_n)
\end{equation}
 of $s_n$ are the orthogonal projections onto $(H \otimes H)^+$ and  $(H \otimes H)^-$, respectively. We shall also often have the occasion to consider the one-dimensional projection
\begin{equation} \label{eq:q}
q_n = \frac{1}{n} \, \sum_{i,j=1}^n e_{ij} \otimes e_{ij}\in M_n(\C) \otimes M_n(\C).
\end{equation}
The range of $q_n$ is the one dimensional subspace spanned by the unit vector
$$\xi = \frac{1}{\sqrt{n}} \, \big( \delta_1 \otimes \delta_1 + \delta_2 \otimes \delta_2 +  \cdots + \delta_n \otimes \delta_n\big),$$
where $(\delta_j)_{j=1}^n$ is the standard orthonormal basis for $\C^n$.
Note that $s_n \xi = \xi$, so $s_nq_n = q_n$, which implies that $q_n \le p_n^+$. We shall often omit the subscript $n$ and write $s$, $p^\pm$ and $q$ for $s_n$, $p_n^\pm$ and $q_n$, respectively.

It is clear that the subspaces $(H \otimes H)^+$ and  $(H \otimes H)^-$ are invariant for $\rho(u):= u \otimes u$, for each unitary $n\times n$ matrix $u$. Let $\rho^+(u)$ and $\rho^-(u)$ denote the restriction of $\rho(u)$ to each of these two invariant subspaces. Then, by the Schur-Weyl duality for the special case of two-tensor factors, $\rho^+$ and $\rho^-$ are irreducible representations of the unitary group ${\mathcal U}(n)$ (see, e.g., \cite{Wey}). They are  obviously not equivalent because $(H\otimes H)^+$ and $(H \otimes H)^-$ have different dimension. It follows that the commutant, $\rho\big({\mathcal U}(n)\big)'$, of $\rho\big({\mathcal U}(n)\big)$ in ${\mathcal B}(H \otimes H)$ is equal to $\C p^+ + \C p^-$. Moreover,
\begin{equation}\label{eq66666667}
E(x) = \int_{{\mathcal U}(n)} \, (u \otimes u) x (u^* \otimes u^*) \, du, \qquad x \in {\mathcal B}(H \otimes H)
\end{equation}
is the trace preserving conditional expectation of ${\mathcal B}(H \otimes H)$ onto the commutant $\rho\big({\mathcal U}(n)\big)' = \C p^+ + \C p^-$. (The integral is with respect to the Haar measure on ${\mathcal U}(n)$.) Being trace preserving, $E$ is the orthogonal projection of ${\mathcal B}(H \otimes H)$ onto $ \C p^+ + \C p^-$ with respect to the Hilbert--Schmidt norm. Using that the Hilbert-Schmidt norm of the projections $p^+$ and $p^-$ is equal to the dimension of $(H \otimes H)^+$ and $(H \otimes H)^-$, respectively, we obtain that
\begin{equation} \label{eq:E}
E(x) = \frac{2 }{n(n+1)}\mathrm{Tr}_n(xp^+) \, p^+ + \frac{2}{n(n-1)} \mathrm{Tr}_n(xp^-) \, p^-, \qquad
x \in {\mathcal B}(H \otimes H),
\end{equation}
where $\mathrm{Tr}_n$ denotes the non-normalized trace on $M_n(\mathbb{C})$.

\begin{defi}\label{def:twirlmap} For $T \in {\mathcal B}(M_n(\C))$ and $u \in {\mathcal U}(n)$, set $\rho_u(T) = \ad(u) \, T \, \ad(u^t)$ and define
\begin{equation} \label{eq:F}
F(T)\colon = \int_{{\mathcal U}(n)} \, \rho_u(T) \, du.
\end{equation}
The map $F \colon {\mathcal B}(M_n(\C)) \to {\mathcal B}(M_n(\C))$ is called the {\bf{twirling map}}.
\end{defi}

Given $u\in {\mathcal U}(n)$, since the adjoint of the transposed $u^t$ of $u$ is $\bar{u}$, we have
$$\rho_u(T)(x) = u T(u^t x \bar{u}) u^*, \qquad x \in M_n(\C).$$
Note that $F(T)$ belongs to the (point-norm) closed  convex hull of $\{\rho_u(T) : u \in {\mathcal U}(n)\}$.

\begin{prop} \label{prop:twirl} \mbox{} The twirling map has the following properties:
\begin{enumerate}
\item $F\big(\mathrm{UCP}(n)\big) \subseteq  \mathrm{UCP}(n)$. \vspace{.1cm}
\item $F\big(\mathrm{UCPT}(n)\big) \subseteq  \mathrm{UCPT}(n)$. \vspace{.1cm}
\item $F\big(\conv(\Aut(M_n(\C)))\big) \subseteq \conv(\Aut(M_n(\C)))$.
\end{enumerate}
\end{prop}

\begin{proof} Items (1) and (2) follow from the fact that the sets $\mathrm{UCP}(n)$ and $\mathrm{UCPT}(n)$ are convex, closed in the point-norm topology and invariant under $\rho_u$  for all $u \in {\mathcal U}(n)$.

(3). By linearity of $F$, it suffices to show that $F(\ad(v))$ belongs to $\conv(\Aut(M_n(\C)))$ for all unitaries $v$ in $M_n(\C)$. Now, $\rho_u(\ad(v)) = \ad(uvu^t)$. Since $\conv(\Aut(M_n(\C)))$ is convex and closed in the point-norm topology, we conclude that $F(\ad(v))$ belongs to $\conv(\Aut(M_n(\C)))$.
\end{proof}

\begin{lemma} The following identity holds:
\begin{equation} \label{eq:twirl1}
\int_{{\mathcal U}(n)} \, u \otimes \bar{u} \, du = \frac{1}{n} \sum_{i,j=1}^n e_{ij} \otimes e_{ij} = q.
\end{equation}
\end{lemma}

\begin{proof} For each $a$ in $M_n(\C)$, let $L_a$ and $R_a$ in ${\mathcal B}(M_n(\C))$ be left and right multiplication by $a$. The map $a \otimes b \mapsto L_a R_{b^t}$ extends to an algebra isomorphism from $M_n(\C) \otimes M_n(\C)$ onto ${\mathcal B}(M_n(\C))$. Applying this isomorphism to \eqref{eq:twirl1} and evaluating at $x \in M_n(\C)$, we see that \eqref{eq:twirl1} is equivalent to
\begin{equation}  \label{eq:twirl2}
\int_{{\mathcal U}(n)} \, uxu^* \, du = \frac{1}{n} \sum_{i,j=1}^n e_{ij} x e_{ji}, \qquad x \in M_n(\C).
\end{equation}
We verify \eqref{eq:twirl2} by showing that both expressions are equal to $\mathrm{Tr}_n(x) \,1_n$. Straightforward calculations show that the trace of both expressions is equal to $\mathrm{Tr}_n(x)$. Next, the left-hand side of \eqref{eq:twirl2} belongs ${\mathcal U}(n)' = \C 1_n$, while the right-hand side of   \eqref{eq:twirl2}  is easily seen to belong to $\{e_{ij} : 1 \le i,j \le n\}' = \C 1_n$. This gives the conclusion.
\end{proof}

For $T \in {\mathcal B}(M_n(\C))$ consider its {\bf{Jamiolkowski transform}}:
$$\widehat{T} = \frac{1}{n} \sum_{i,j=1}^n T(e_{ij}) \otimes e_{ij} \in M_n(\C) \otimes M_n(\C).$$
It is well-known, see, e.g., \cite[Proposition 1.5.4]{BrOz}, that $T$ is completely positive if and only if $\widehat{T}$ is positive.
Part (2) of the lemma below shows that the Jamiolkowski transform intertwines the conditional expectation $E$ and the twirl map $F$ from \eqref{eq:F} and $(\ref{eq66666667})$.

\begin{lemma} \label{lm:intertwining}
The following hold for each $T \in {\mathcal B}(M_n(\C))$:
\begin{enumerate}
\item $\ad(u \otimes u)(\widehat{T}) = \widehat{\rho_u(T)}$, for all $u \in {\mathcal U}(n)$.\vspace{0.1cm}
\item $\widehat{F(T)} = E(\widehat{T})$.
\end{enumerate}
\end{lemma}

\begin{proof} (1). If we apply $T \otimes \mathrm{id}_{M_n(\mathbb{C})}$ to \eqref{eq:twirl1} we find that
\begin{equation} \label{eq:T-hat}
\widehat{T} = \int_{{\mathcal U}(n)} \, T(v) \otimes \bar{v} \, dv.
\end{equation}
Hence,
\begin{eqnarray*}
 \widehat{\rho_u(T)} & =&  \int_{{\mathcal U}(n)} \, \rho_u(T)(v) \otimes \bar{v} \, dv  \;
= \; \int_{{\mathcal U}(n)} u T(u^t v \bar{u}) u^* \otimes \bar{v} \, dv \\
&=& \int_{{\mathcal U}(n)} \, u T(w)u^* \otimes u\bar{w}u^* \, d w \; = \; \ad(u \otimes u) \Big( \int_{{\mathcal U}(n)} \, T(w) \otimes \bar{w} \, dw\Big) \; = \; \ad(u \otimes u)(\widehat{T}),
\end{eqnarray*}
as desired. (At the third equality sign we used the substitution $w = u^tv\bar{u}$ and invariance of the Haar measure, and \eqref{eq:T-hat} is used at the last equality sign.)

(2). It follows from (1) that
$$\widehat{F(T)} = \int_{{\mathcal U}(n)} \, \widehat{\rho_u(T)} \, du = \int_{{\mathcal U}(n)} \, (u \otimes u) \, \widehat{T} \, (u^* \otimes u^*) \, du = E(\widehat{T}),$$  as claimed.\end{proof}

\noindent
For each integer $n \ge 2$, recall the {\bf{Holevo--Werner channels}} $W_n^+, W_n^- \in {\mathcal B}(M_n(\C))$, studied in \cite{MWo}:
\begin{equation} \label{eq:HW}
W_n^+(x) = \frac{1}{n+1} \Big(\mathrm{Tr}_n(x) \, 1_n + x^t\Big), \qquad W_n^-(x) = \frac{1}{n-1} \Big(\mathrm{Tr}_n(x) \, 1_n - x^t\Big), \quad x \in M_n(\C).
\end{equation}
They can alternatively be expressed as
\begin{equation}  \label{eq:HW2}
W_n^+(x) = \frac{1}{2n+2} \sum_{i,j=1}^n (e_{ij}+e_{ji})x(e_{ij}+e_{ji})^*, \quad W_n^-(x) = \frac{1}{2n-2} \sum_{i,j=1}^n (e_{ij}-e_{ji})x(e_{ij}-e_{ji})^*,
\end{equation}
for $x\in M_n(\mathbb{C})$\,. (One can easily verify \eqref{eq:HW2} by first considering the case where $x = e_{k\ell}$ is a matrix unit.) We conclude by \eqref{eq:HW2}  that $W_n^+$ and $W_n^-$ are UCPT$(n)$-maps. Using notation set-forth above  (cf. $(\ref{eq:PS})$, $(\ref{eq:q})$), the Jamiolskowski transforms of the Holevo--Werner channels and of the identity operator are
\begin{equation} \label{eq:HW1}
\widehat{W_n^+} = \frac{2}{n(n+1)} \, p^+, \qquad \widehat{W_n^-} = \frac{2}{n(n-1)} \, p^-, \qquad \widehat{\id_n} = q.
\end{equation}

Recall that the $2$-norm on $M_n(\C)$ is defined by $\|x\|_2 = \tau_n(x^*x)^{1/2}$, $x \in M_n(\C)$\,.
As already observed in \cite{VoW}, the twirling map $F$ is a projection of $M_n(\C)$ onto the subspace spanned by $W_n^+$ and $W_n^-$, and it maps $\mathrm{UCP}(n)$ onto the line segment spanned by $W_n^+$ and $W_n^-$\,. More precisely,

\begin{theorem} \label{thm:HW}
The following hold for all $n\geq 2$:
\begin{enumerate}
\item $F(W_n^+) = W_n^+$ and $F(W_n^-) = W_n^-$. \vspace{.2cm}
\item $F(T) = \Tr_n(\widehat{T} \, p^+) \, W_n^+ + \Tr_n(\widehat{T} \, p^-) \, W_n^-$, for all $T \in {\mathcal B}(M_n(\C))$. \vspace{.2cm}
\item If $T \in \mathrm{CP}(n)$ has Choi representation $T(x) = \sum_{i=1}^d a_i x a_i^*$,  $x \in M_n(\C)$, where $d\in \mathbb{N}$ and $a_1\,, \ldots \,, a_d\in M_n(\mathbb{C})$, then
$$F(T) = c^+(T) \, W_n^+ + c^-(T) \, W_n^-,$$
where the coefficients $c^+(T)$ and $c^-(T)$ are given by
$$c^+(T) = \frac14 \sum_{i=1}^d \|a_i+a_i^t\|_2^2, \qquad c^-(T) = \frac14 \sum_{i=1}^d \|a_i-a_i^t\|_2^2.$$
\end{enumerate}
\end{theorem}

\begin{proof} (1). An easy calculation shows that $\rho_u(W_n^\pm) = W_n^\pm$, for all $u \in {\mathcal U}(n)$. Therefore (1) holds.

(2). From Lemma \ref{lm:intertwining} together with \eqref{eq:E}, and \eqref{eq:HW1},  we deduce that
$$\widehat{F(T)} = E(\widehat{T}) = \frac{2}{n(n+1)} \mathrm{Tr}_n(\widehat{T} p^+) \, p^+ + \frac{2}{n(n-1)} \mathrm{Tr}_n(\widehat{T} p^-) \, p^- = \mathrm{Tr}_n(\widehat{T} p^+) \, \widehat{W_n^+} + \mathrm{Tr}_n(\widehat{T} p^-) \, \widehat{W_n^-} .$$
Since the map $T \mapsto \widehat{T}$ is linear and injective, we conclude that (2) holds.

(3). Note first that it suffices to consider the case $d=1$. We can therefore assume that $T(x) = axa^*$, $x\in M_n(\mathbb{C})$, for some $a \in M_n(\C)$. In this case,
$\widehat{T} = ({1}/{n}) \sum_{i,j=1}^n ae_{ij}a^* \otimes e_{ij}\,.$
Hence
\begin{equation} \label{eq:Tr-1}
\mathrm{Tr}_n(\widehat{T}) = \frac{1}{n} \sum_{i,j=1}^n \mathrm{Tr}_n(ae_{ij}a^*) \, \mathrm{Tr}_n(e_{ij}) =  \frac{1}{n} \sum_{i=1}^n \mathrm{Tr}_n(ae_{ii}a^*) = \tau_n(aa^*).
\end{equation}
Let $s = s_n$ be the flip symmetry defined above, and write $a = (a_{ij})_{1\leq i, j\leq n}$. Then
\begin{equation} \label{eq:Tr-2}
\mathrm{Tr}_n(\widehat{T}s) =  \frac{1}{n} \sum_{i,j,k,\ell=1}^n \mathrm{Tr}_n(ae_{ij}a^*e_{k\ell}) \, \mathrm{Tr}_n(e_{ij}e_{\ell k})
=\frac{1}{n} \sum_{i,j=1}^n \mathrm{Tr}_n(ae_{ij}a^*e_{ij}) = \frac{1}{n} \sum_{i,j=1}^n a_{ji} \bar{a}_{ij} = \tau_n(a \bar{a}).
\end{equation}
Now use item (2) together with \eqref{eq:PS}, \eqref{eq:Tr-1} and \eqref{eq:Tr-2} to conclude that
\begin{eqnarray*}
F(T) &=& \mathrm{Tr}_n(\widehat{T}\, p^+) \, W_n^+ + \mathrm{Tr}_n(\widehat{T}\, p^-) \, W_n^- \\
&=& \frac12 \tau_n(aa^* + a\bar{a}) \, W_n^+ + \frac12 \tau_n(aa^* - a\bar{a}) \, W_n^- \\
&=& \frac14 \|a+a^t\|_2^2 \; W_n^+ + \frac14 \|a-a^t\|_2^2 \; W_n^-.
\end{eqnarray*}
In the last equality we have used that transposition is trace preserving, along with the identities $a^ta^* = (\bar{a}a)^t$ and $a^t (a^t)^* = (a^*a)^t$.
\end{proof}

\begin{cor} \label{cor:HW1}
Let $T$ be a UCP$(n)$-map written in Choi form as $T(x) = \sum_{i=1}^d a_i x a_i^*$, $x \in M_n(\C)$, for some $d\in \mathbb{N}$, $a_i \in M_n(\C)$, $1\leq i\leq d$.
\begin{enumerate}
\item If all $a_i$ are symmetric, i.e., $a_i^t = a_i$, $1\leq i\leq d$, then $F(T) = W_n^+$.\vspace{.1cm}
\item If all $a_i$ are anti-symmetric, i.e., $a_i^t = -a_i$, $1\leq i\leq d$, then $F(T) = W_n^-$.
\end{enumerate}
\end{cor}

\begin{proof}  $(1)$. If all $a_i$ are symmetric, then $c^-(T) = 0$, in which case by Theorem \ref{thm:HW} (3) it follows that $F(T) = c^+(T) \, W_n^+$. Use now that $F(T)$ and $W_n^+$ are unital to conclude that $c^+(T) = 1$. Item (2) is proved similarly.
\end{proof}

\begin{cor}[Mendl--Wolf, \cite{MWo}] \label{cor:MWmixture} \mbox{} \vspace{.1cm}
\begin{enumerate}
\item $W_n^+ \in \conv\big(\Aut(M_n(\C))\big)$, for all integers $n \ge 2$. \vspace{.1cm}
\item $W_n^- \in \conv\big(\Aut(M_n(\C))\big)$, for all even integers $n \ge 2$.
\end{enumerate}
\end{cor}

\begin{proof}
(1).  It follows from Corollary \ref{cor:HW1} (1), with $T = \id_n$, $d=1$, and $a_1 = 1_n = a_1^t$ that \[ W_n^+ = F(\id_n). \]
This proves the claim because $F(\id_n)$ belongs to $\conv\big(\Aut(M_n(\C))\big)$ by  Proposition \ref{prop:twirl} (3).

(2). For each even integer $n \ge 2$, there is an anti-symmetric unitary $v$ in $M_n(\C)$. Take, for example,
$$v = \begin{pmatrix} 0 & 1_k \\ -1_k & 0 \end{pmatrix} \in M_n(\C),$$
where $n = 2k$. It follows from Corollary \ref{cor:HW1} (2), with $T = \ad(v)$, $d=1$, and $a_1 = v$ that \[ W_n^- = F(\ad(v)). \]
Furthermore, $F(\ad(v))$ belongs to $\conv\big(\Aut(M_n(\C))\big)$ by  Proposition \ref{prop:twirl} (3).
\end{proof}

\begin{lemma} \label{dist=2}
$\|W_n^+-W_n^-\|_\cb=2$ for all $n \ge 2$.
\end{lemma}

\begin{proof} Since $W_n^+$ and $W_n^-$ are UCP-maps, they are complete contractions, and hence $\|W_n^+-W_n^-\|_\cb \le 2.$
To prove the other inequality note first that
$$
W_n^+(x) - W_n^-(x) = \frac{2n}{n^2-1} \left(x^t - \frac{1}{n} \, \mathrm{Tr}_n(x) \, 1_n\right).
$$
Let $s=s_n$ be the flip symmetry defined in \eqref{eq:S} and let $q = q_n$ be the projection defined in \eqref{eq:q}. Then, by the identity above,
$$\|W_n^+-W_n^-\|_\cb \ge \big\|\big((W_n^+-W_n^-) \otimes \id_{M_n(\C)}\big)(s)\big\| = \frac{2n}{n^2-1} \left\|nq-\frac{1}{n} 1_n\right\| = \frac{2n}{n^2-1} \, \left(n-\frac{1}{n}\right) = 2,$$
thus giving the conclusion.
\end{proof}

\begin{lemma}[Mendl--Wolf, \cite{MWo}] \label{lm:MWmin}
For all odd integers $n \ge 1$,
$$\min_{v \in {\mathcal U}(n)} \, \left\| \frac{v+v^t}{2} \right\|_2^2 = \frac{1}{n}.$$
\end{lemma}

\begin{proof} In \cite{MWo} (see Theorem 13 and its proof) it was verified that
\begin{equation} \label{eq:MW}
\min_{v \in {\mathcal U}(n)} \, \tau(v \bar{v}) = \frac{2}{n} -1.
\end{equation}
Since $\|v+v^t\|_2^2 = 2 + 2\tau(v\bar{v})$, formula \eqref{eq:MW} is equivalent to the identity we wish to verify. For the convenience of the reader, we include an elementary proof of the lemma.

Let $v \in {\mathcal U}(n)$ and set $a = (v+v^t)/2$ and $b =(v-v^t)/2$. Then $v = a+b$, $a^t=a$, and $b^t = -b$. Since $n$ is odd and $\det(b) = \det(b^t) = (-1)^n \det(b)$, we conclude that $\det(b) = 0$.
Hence $b \xi = 0$, for some unit vector $\xi \in \C^n$. Thus $\|a\xi\| = \|v \xi\| = 1$, so $\|a\| \ge 1$. It follows that
$$\left\| \frac{v+v^t}{2} \right\|_2^2  = \|a\|_2^2 = \frac{1}{n} \mathrm{Tr}_n(a^*a) \ge \frac{1}{n}.$$

To prove the reverse inequality consider the unitary
\begin{equation} \label{eq:v}
v = e_{11} + (e_{23} - e_{32}) + (e_{45}-e_{54}) + \cdots + (e_{n-1,n}-e_{n,n-1}).
\end{equation}
Then $v+v^t = 2e_{11}$, so
$\| (v+v^t)/{2}\|_2^2 = \|e_{11}\|_2^2 = {1}/{n},$ which completes the proof.
\end{proof}

\begin{theorem} For each odd integer $n \ge 3$,
$$d_{\text{cb}}\big(W_n^-, \conv\big(\Aut(M_n(\C))\big)\big) = {2}/{n}.$$
\end{theorem}

\begin{proof}  Let $v \in {\mathcal U}(n)$ be such that $\|(v+v^t)/2\|_2^2 = 1/n$, cf.\ Lemma \ref{lm:MWmin} or \eqref{eq:v}. Since $\|v\|_2=\|v^t\|_2=1$, it follows from the parallelogram identity that $\|(v-v^t)/2\|_2^2 = (n-1)/n$. By Theorem \ref{thm:HW} (3),
$$F(\ad(v)) = \frac{1}{n} \, W_n^+ + \frac{n-1}{n} \, W_n^-.$$
We know from Proposition \ref{prop:twirl} that $F(\ad(v))$ belongs to $\conv\big(\Aut(M_n(\C))\big)$, so by Lemma \ref{dist=2},
$$d_{\text{cb}}\big(W_n^-, \conv\big(\Aut(M_n(\C))\big)\big) \le \left\|W_n^- - \left(\frac{1}{n} \, W_n^+ + \frac{n-1}{n} \, W_n^-\right)\right\|_\cb = \frac{1}{n} \|W_n^- - W_n^+\|_\cb = \frac{2}{n}.$$

Let now $v$ be any unitary in $M_n(\C)$. The same reasoning as above shows that $F(\text{ad}(v)) = \lambda W_n^+ + (1-\lambda) W_n^-$, where $\lambda = \|(v+v^t)/2\|_2^2$,  and it follows from Lemma \ref{lm:MWmin} that ${1}/{n} \le \lambda \le 1$. Fix $T$ in $\conv\big(\Aut(M_n(\C))\big)$. By convexity of the line segment
$\big\{\lambda W_n^+ + (1-\lambda) W_n^- : {1}/{n} \le \lambda \le 1\big\}$,
we see that $F(T) =  \lambda W_n^+ + (1-\lambda) W_n^-,$ for some $1/n\leq \lambda \leq 1$. It follows that
\begin{eqnarray*}
 \|W_n^--T\|_\cb & \ge &  \|F(W_n^-) - F(T)\|_\cb \\
&=&  \|W_n^- - (\lambda W_n^+ + (1-\lambda) W_n^-)\|_\cb \\ &=& \lambda \|W_n^- - W_n^+\|_\cb \: = \; 2\lambda \; \ge \; {2}/{n},
\end{eqnarray*}
wherein we have used Lemma \ref{dist=2}. As $T \in \conv\big(\Aut(M_n(\C))\big)$ was arbitrarily chosen, we conclude that
$d_{\text{cb}}\big(W_n^-,  \conv\big(\Aut(M_n(\C))\big)\big) \ge {2}/{n}\,,$
as wanted.
\end{proof}

The corollary below follows immediately from the theorem above and its proof.

\begin{cor}[Mendl--Wolf, \cite{MWo}]\label{hwmixtureunit}
For each odd integer $n \ge 1$ and for $0 \le \lambda \le 1$,
$$\lambda W_n^+ + (1-\lambda)W_n^- \in \conv\big(\Aut(M_n(\C))\big)$$
if and only if $\lambda \ge 1/n$.
\end{cor}

\section{Factorizability of the Holevo--Werner Channels}

It was shown in Corollary \ref{cor:MWmixture} that the Holevo--Werner channel $W_n^+$, for all integers $n \ge 3$, and $W_n^-$ for all even integers $n \ge 4$, belong to $\conv\big(\Aut(M_n(\C))\big)$, and hence they are factorizable. Also, it was shown in \cite[Example 3.1]{HM6} that $W_3^-$ is not factorizable. We shall prove here that the Holevo--Werner channels $W_n^-$ are factorizable of degree 4, for all odd integers $n \ge 5$. Furthermore, we shall discuss factorizability of convex combinations of $W_3^+$ and $W_3^-$, and determine the cb-distance from $W_3^-$ to the factorizable maps. Keeping the notation from \cite{HM6}, we denote by ${\mathcal F}{\mathcal M}(M_n(\mathbb{C}))$ the set of factorizable UCPT$(n)$-maps.

\begin{lemma} \label{lm:5unitaries}
There exists five self-adjoint unitaries $v_1,v_2,v_3,v_4,v_5$ in $M_4(\C)$ such that
\begin{enumerate}
\item $v_iv_j + v_jv_i=0$, when $i \ne j$ (anti-commute), \vspace{.1cm}
\item $\{v_iv_j : 1 \le i < j \le 5\}$ is an orthonormal set in $M_4(\C)$ with respect to the inner product arising from the normalized trace $\tau_4$ on $M_4(\C)$.
\end{enumerate}
\end{lemma}

\begin{proof} This follows from standard Clifford algebra techniques. Consider the $2 \times 2$ matrices
$$J = \begin{pmatrix} i & 0 \\ 0 & -i \end{pmatrix}, \qquad
K = \begin{pmatrix} 0 & 1 \\ -1 & 0 \end{pmatrix}, \qquad
L = \begin{pmatrix} 0 & i \\ i & 0 \end{pmatrix}.$$
Check that $J$, $K$, and $L$ are anti-commuting skew-adjoint unitaries which satisfy the relations $JK=L$, $KL=J$, and $LJ=K$. In particular, $\{1_2, J,K,L\}$ is  an orthonormal basis for $M_2(\C)$ with respect to the inner product arising from the normalized trace $\tau_2$ on $M_2(\C)$.
Use these relations to see that the following five $4 \times 4$ matrices
$$v_1 = \begin{pmatrix} 1_2 & 0 \\ 0 & -1_2 \end{pmatrix}, \qquad v_2 = \begin{pmatrix} 0 & 1_2 \\ 1_2 & 0 \end{pmatrix}, $$
$$v_3 = \begin{pmatrix} 0 &-J \\ J & 0 \end{pmatrix}, \qquad
v_4 = \begin{pmatrix} 0 &-K\\ K & 0 \end{pmatrix}, \qquad
v_5 = \begin{pmatrix} 0 &-L \\ L& 0 \end{pmatrix}.
$$
have the desired properties.
\end{proof}

\begin{theorem} \label{thm:1111} \mbox{} The following hold:
\begin{enumerate}
\item $W_5^-$ has an exact factorization through $M_5(\C) \otimes M_4(\C)$. \vspace{.1cm}
\item $W_n^-$ is factorizable of degree 4, for all odd integers $n \ge 5$.
\end{enumerate}
\end{theorem}

\begin{proof}
(1). Let $\sigma = (\sigma_{ij})_{1\leq i,j\leq 5}$ be a unitary matrix in $M_5(\C)$ which is zero on the diagonal and such that all off-diagonal entries have modulus $1/2$. For example, one can consider
$$\sigma = \frac12 \begin{pmatrix} 0 & \alpha & \beta & \beta & \alpha \\ \alpha & 0 & \alpha & \beta & \beta \\ \beta & \alpha & 0 & \alpha & \beta \\ \beta & \beta & \alpha & 0 & \alpha \\ \alpha & \beta & \beta & \alpha & 0 \end{pmatrix},$$
where $\alpha = -1/2 + i {\sqrt{3}}/{2}$ and $\beta = -1/2 + i {\sqrt{3}}/{2}$. Use that  $|\alpha| = |\beta| = 1$ and $\mathrm{Re}(\alpha \bar{\beta})= -1/2$ to verify that $\sigma$ has the desired properties. Further, let $v_1, \dots, v_5$ be as in Lemma \ref{lm:5unitaries} and define a unitary $u$ by
\begin{equation}\label{eq2222222210}
u = \begin{pmatrix} u_{11} & u_{12} & \cdots & u_{15} \\ u_{21} & u_{22} & \cdots & u_{25} \\ \vdots & \vdots & & \vdots\\ u_{51} & u_{52} & \cdots & u_{55} \end{pmatrix}
:= \begin{pmatrix} v_1 & 0 & \cdots & 0 \\ 0 & v_2 &  & 0 \\ \vdots & & \ddots & \\ 0 & 0 & & v_5 \end{pmatrix}(\sigma \otimes 1_4)  \begin{pmatrix} v_1 & 0 & \cdots & 0 \\ 0 & v_2 &  & 0 \\ \vdots & & \ddots & \\ 0 & 0 & & v_5 \end{pmatrix},\end{equation}
where the block matrix entries $u_{ij}$ belong to $M_4(\C)$. We will show that
\begin{equation} \label{eq:T=W}
W_5^-(x) = (\id_5 \otimes \tau_4) \big(u(x \otimes 1_4)u^*\big), \qquad x \in M_5(\C),
\end{equation}
thus proving the assertion that $W_5^-$  has an exact factorization through $M_5(\C) \otimes M_4(\C)$.

Observe first that
$$
u_{ij} = \sigma_{ij} v_iv_j, \qquad 1 \le i,j \le 5.
$$
Since $\sigma_{jj}=0$, for all $j$, and the $v_j$'s anti-commute, we see that $u_{ij} = - u_{ji}$, for all $1\leq i,j\leq 5$. Consequently, we can write
\begin{equation} \label{eq:654}
u = \sum_{1 \le i < j \le 5} a_{ij} \otimes u_{ij},
\end{equation}
where $a_{ij} = e_{ij}-e_{ji}$, for $1 \le i < j \le 5$, and where $(e_{ij})_{1\leq i, j\leq 5}$ are the matrix units in $M_5(\C)$.

Recall from  Lemma \ref{lm:5unitaries} (2) that $\{v_iv_j\}_{1 \le i < j \le 5}$ is an orthonormal set in $M_4(\C)$ with respect to the inner product arising from the normalized trace $\tau_4$. Using this fact together with \eqref{eq:654} and \eqref{eq:HW2}, we can conclude that for all $x \in M_5(\C)$,
\begin{eqnarray*}
 (\id_5 \otimes \tau_4) \big(u(x \otimes 1_4)u^*\big) &=& \sum_{1 \le i < j \le 5} \; \;\sum_{1 \le k < \ell \le 5} \tau_4(u_{ij}u_{k\ell}^*) \, a_{ij}xa_{k\ell}^* \\
&=&  \sum_{1 \le i < j \le 5} \: |\sigma_{ij}|^2 a_{ij}xa_{ij}^*  \; = \;  \frac14\sum_{1 \le i < j \le 5} \:  a_{ij}xa_{ij}^* \; = \;  W_5^-(x),
\end{eqnarray*}
This proves item (1).

(2). It follows from (1) that (2) holds for $n=5$. Suppose now that $n \ge 7$ is an odd integer and
set $k= (n-5)/2$. Define $R \in \mathrm{UCPT}(n)$ by
$$R(x) = \begin{pmatrix} W_5^-(x_{11}) & 0 \\ 0 & W_{2k}^-(x_{22}) \end{pmatrix}, \qquad x = \begin{pmatrix} x_{11} & x_{12} \\ x_{21} & x_{22} \end{pmatrix} \in M_n(\C),$$
where the block matrix decomposition of $x$ is taken with respect to the decomposition $\C^n = \C^5 \oplus \C^{2k}$, so that $x_{11} \in M_5(\C)$ and $x_{22} \in M_{2k}(\C)$. By Corollary \ref{cor:MWmixture} (2), $W_{2k}^-\in \conv\big(\Aut(M_{2k}(\C))\big)$, so
there exist an integer $s \ge 1$, unitaries $u_1, \dots, u_s$ in $M_{2k}(\C)$, and positive scalars $c_1, \dots, c_s$ with $\sum_{i=1}^s c_i = 1$, such that
$$W_{2k}^- = \sum_{i=1}^s c_i \,  \ad(u_i).$$

For $1\leq i\leq s$, define unitaries $u_i^+$ and $u_i^-$ in $M_n(\C)$ by
$$u_i^+ = \begin{pmatrix} u & 0 \\ 0 & u_i \otimes 1_4 \end{pmatrix}, \qquad
u_i^- = \begin{pmatrix} u & 0 \\ 0 & -u_i \otimes 1_4 \end{pmatrix},$$
where $u$ is the unitary defined by (\ref{eq2222222210}) above. Further define $R^+, R^- \in \mathrm{UCPT}(n)$ by
$$R^\pm(x) = \sum_{i=1}^s c_i \, (\id_n \otimes \tau_4) \big(u_i^\pm (x \otimes 1_4) (u_i^\pm)^*\big), \qquad x \in M_n(\C).$$
Then $R =(R^+ + R^-)/2$, and hence $R$ is factorizable of degree 4. As before, let $(e_{ij})_{1\leq i,j\leq n}$ be the matrix units in $M_n(\C)$ and set $a_{ij} = e_{ij} - e_{ji}$ for $1 \le i < j \le n$. Then, by \eqref{eq:HW2},
$$R(x) \:= \: \begin{pmatrix} W_5^-(x_{11}) & 0 \\ 0 & W_{2k}^-(x_{22}) \end{pmatrix} \:= \: \frac14 \sum_{1 \le i < j \le 5} a_{ij} x a_{ij}^* + \frac{1}{2k-1} \sum_{6 \le i < j \le n} a_{ij} x a_{ij}^*$$
Since $a_{ij}^t = -a_{ij}$, for all $i, j$, it follows from Corollary \ref{cor:HW1} (2) that
$$W_n^- = F(R) = \int_{{\mathcal U}(n)} \, \rho_u( R )\, du.$$
The map $\rho_u(R) = \ad(u) \, R \, \ad(u^t)$ is factorizable of degree $4$ for each $u \in {\mathcal U}(n)$, and hence so is  $W_n^-$.
\end{proof}

We will need a few intermediate lemmas before we can prove Theorem \ref{thm12} below. Given a finite von Neumann algebra $N$ with normal faithful trace $\tau_N$ and  $1\leq p< \infty$, we shall consider the $p$-norm of elements in $M_3(N)$ defined as follows:
$$\|x\|_p = (\tau_3 \otimes \tau_N)\big((x^*x)^{\frac{p}{2}}\big)^{{1}/{p}}\,, \quad x\in M_3(N)\,. $$

\begin{lemma} \label{lemma13}
Let $N$ be a finite von Neumann algebra with normal faithful trace $\tau_N$, and let
$$u = \big(u_{ij}\big)_{1\leq i,j\leq 3} \in M_3(N), \qquad u_{ij} \in N,$$
be a unitary operator. Let $u^T = (u_{ji})_{i, j} \in M_3(N)$ be the transpose of $u$, and set $b = (u-u^T)/2$. Then
\begin{enumerate}
\item $\|b\| \le 5/3$,
\item $\|b\|_2^2 \le \|b\|_1$,
\item $\|b\|_4^4 \ge (3/2) \, \|b\|_2^4$.
\end{enumerate}
\end{lemma}

\begin{proof}
(1). Denote the transposition map $x \mapsto x^t$ on $M_3(\C)$ by $t_3$, so that $u^t = (t_3 \otimes \id_N)(u)$ and $b = (1/2)\, \big((\id_3 - t_3) \otimes \id_N\big)(u)$. It suffices to show that
\begin{equation} \label{eq2341}
\|\id_3 - t_3\|_\cb \le {10}/{3}.
\end{equation}
To prove \eqref{eq2341}, we first show that  $W_3^+ - (1/6) \, \id_3$ is a completely positive map. For this it suffices to show that its Jamiolkowski transform, $\widehat{W_3^+} - (1/6) \, \widehat{\id_3}$ is a positive operator in $M_3(\C) \otimes M_3(\C)$, cf.\ \cite[Proposition 1.5.4]{BrOz}. We know from \eqref{eq:HW1} that
$$\widehat{W_3^+} = \frac16 \, p^+ , \qquad \widehat{\id_3} = \frac13 \, \sum_{i,j=1}^3 e_{ij} \otimes e_{ij} = q$$
where $p^+ = p_3^+$ and  $q = q_3$  are the projection in $M_3(\C) \otimes M_3(\C)$ defined in \eqref{eq:S} and \eqref{eq:q}. It was observed right after \eqref{eq:q} that $q \le p$, and so $\widehat{W_3^+} - (1/6) \, \widehat{\id_3}  \ge 0$.
This proves that $W_3^+ - (1/6) \, \id_3$ is completely positive.
Furthermore, it follows from the definition of the Holevo--Werner channels in \eqref{eq:HW}, that $2 t_3 = 4W_3^+-2W_3^-$. Thus
$\id_3 - t_3 = \id_3 + W_3^- - 2W_3^+ = \left((2/3) \, \id_3 + W_3^-\right) - 2\left(W_3^+-(1/6) \, \id_3\right)$,
and hence
$$\|\id_3 - t_3 \|_\cb \le \left\|\frac23 \, \id_3 + W_3^-\right\|_\cb + 2 \left\|W_3^+-\frac16 \, \id_3\right\|_\cb \le \frac53 + 2 \cdot \frac56 = \frac{10}{3},$$
because $\|T\|_\cb = \|T(1)\|$ for every completely positive map $T$.

(2). Note that $\|u^t\|_2^2 = (1/3) \, \sum_{i,j=1}^3 \|u_{ji}\|_2^2 = (1/3) \, \sum_{i,j=1}^3 \|u_{ij}\|_2^2= \|u\|_2^2\,.$
Since $(u^Tu^*)^* = u(u^T)^*$, it follows that
\begin{eqnarray*}
\mathrm{Re}(\tau_3 \otimes \tau_N)\big((u-b)b^*\big) &=& (1/4) \, \mathrm{Re}(\tau_3 \otimes \tau_N)\big((u+u^t)(u-u^t)^*\big) \\
&=&  (1/4) \Big(\|u\|_2^2 - \|u^t\|_2^2 +  \mathrm{Re}(\tau_3 \otimes \tau_N)\big(u(u^t)^* -u^t u^*\big) \Big)= 0.
\end{eqnarray*}
We conclude that
$$0 = \mathrm{Re}(\tau_3 \otimes \tau_N)\big((u-b)b^*\big) = \mathrm{Re}(\tau_3 \otimes \tau_N)(ub^*) - \|b\|_2^2 \le \|ub^*\|_1 - \|b\|_2^2 = \|b\|_1-\|b\|_2^2,$$
which proves (2).

(3). The element $b \in M_3(N)$  has the following matrix representation
$$b = \begin{pmatrix} 0 & z & -y \\ -z & 0 & x \\ y & -x & 0 \end{pmatrix}, \quad \text{where} \quad x = \frac12 \, (u_{23} -u_{32}), \qquad y = \frac12 \, (u_{31}-u_{13}), \qquad z = \frac12 \, (u_{12}-u_{21}),$$
and $\|b\|_2^2 = (2/3) \, \big( \|x\|_2^2 + \|y\|_2^2 + \|z\|_2^2\big).$
Moreover,
$$b^*b = \begin{pmatrix} y^*y+z^*z & -y^*x & -z^*x \\ -x^*y & z^*z+x^*x & -z^*y \\ -x^*z & -y^*z & x^*x+y^*y
\end{pmatrix}.$$
Hence
\begin{eqnarray*}
\|b\|_4^4 & = & \|b^*b\|_2^2 \\
&=& \frac13\Big( \|y^*y+z^*z\|_2^2 + \|z^*z+x^*x\|_2^2 + \|x^*x+y^*y\|_2^2 + 2\|x^*y\|_2^2 + 2\|y^*z\|_2^2 + 2\|z^*x\|_2^2\Big) \\
&=& \frac23\Big( \|x^*x\|_2^2 + \|y^*y\|_2^2 + \|z^*z\|_2^2 + \tau_N(y^*yz^*z) + \tau_N(z^*zx^*x) + \tau_N(x^*xy^*y) \\ && \qquad +\tau_N(xx^*yy^*) + \tau_N(yy^*zz^*) + \tau_N(zz^*xx^*)\Big) \\
&=& \frac13 \, \Big(\|x^*x+y^*y+z^*z\|_2^2 + \|xx^*+yy^*+zz^*\|_2^2\Big) \\
&\ge& \frac13 \, \Big(\tau_N(x^*x+y^*y+z^*z)^2 + \tau_N(xx^*+yy^*+zz^*)^2\Big) \\
&=& \frac23 \, \Big(\|x\|_2^2 + \|y\|_2^2 + \|z\|_2^2\Big)^2
\; =\; \frac32 \, \|b\|_2^4.
\end{eqnarray*}
Along the way we have used Cauchy--Schwartz inequality, which in our context asserts that $|\tau_N(a)| \le \|a\|_2 \cdot \|1_N\|_2 = \|a\|_2$, for all $a \in N$.
\end{proof}

\begin{lemma} \label{lemma14}
Let $(N,\tau_N)$, $u \in M_3(N)$, and $b  = (u-u^t)/2$ be as in Lemma \ref{lemma13}. Then $\|b\|_2^2 \le 25/27$.
\end{lemma}

\begin{proof} Denote the normal faithful tracial state $\tau_3 \otimes \tau_N$ on $M_3(N)$ by $\tilde{\tau}$. Consider the positive element $h = (6/5)\, |b|$ in $M_3(N)$. Then, by Lemma \ref{lemma13},
\begin{equation} \label{eq637}
 0 \le h \le 2I, \qquad \tilde{\tau}(h^2) \le (6/5) \, \tilde{\tau}(h), \qquad \tilde{\tau}(h^4) \ge (3/2) \,  \tilde{\tau}(h^2)^2.
\end{equation}
Consider the spectral resolution
$h = \int_0^2 \lambda \, dE(\lambda)$ of $h$, and define a probability measure $\mu$ on $[0,2]$ by $\mu = \tilde{\tau} \circ E$. Then
$$\tilde{\tau}(h^n) = \int_0^2 t^n \, d\mu(t), \qquad n \ge 0.$$
The polynomial $p(w) = (w-1)^2(w-2)(w+4) = w^4 -11 w^2 + 18 w -8$, $w \in \R,$ is negative on $[0,2]$, so
\begin{equation}\label{eq9090}\tilde{\tau}(h^4) -11\tilde{\tau}(h^2) + 18\tilde{\tau}(h) - 8 = \tilde{\tau}(p(h)) = \int_0^2 p(w) \, d\mu(w)   \le 0.\end{equation}
 Set $\alpha = \tilde{\tau}(h^2)$. Then $\tilde{\tau}(h) \ge (5/6) \, \alpha$ and $\tilde{\tau}(h^4) \ge (3/2) \, \alpha^2$  by \eqref{eq637}. Inserting $\alpha$ into $(\ref{eq9090})$ yields
$(3/2) \alpha^2 - 11 \alpha +18 \cdot(5/6) \, \alpha - 8 \le 0,$
which implies that $\tilde{\tau}(h^2) = \alpha \le 4/3$. This shows that
$$\|b\|_2^2 = \left(\frac56\right)^2 \, \tilde{\tau}(h^2) \le \left(\frac56\right)^2 \cdot\,\frac43 = \frac{25}{27},$$
as desired.
\end{proof}

%

The following lemma gives a generalization of Theorem \ref{thm:HW} (3).

\begin{lemma} \label{lemma16}
Let $n \ge 1$ be an integer, $N$ a finite von Neumann algebra with a normal faithful tracial state $\tau_N$, and let $a =(a_{ij})_{1\leq i,j\leq n} \in M_n(N)$. Define $T_a \in \mathrm{CP}(n)$ by $T_a(x) = (\id_n \otimes \tau_N) \big(a(x \otimes 1_N) a^*\big)$, $x \in M_n(\C).$
Then
$$F(T_a) = ({1}/{4}) \, \|a+a^t\|_2^2 \; W_n^+ + ({1}/{4}) \, \|a-a^t\|_2^2 \; W_n^-,$$\
where $a^t = (a_{ji})_{1\leq i,j\leq n} \in M_n(N)$ is the transpose of $a$.
\end{lemma}

\begin{proof} We compute the Jamiolkowski transform of $T_a$. Write $a = \sum_{k, \ell=1}^n e_{k\ell} \otimes a_{k\ell}$, where $a_{k\ell} \in N$. Then
\begin{eqnarray*}
\widehat{T_a} & = &  \frac{1}{n} \, \sum_{i,j=1}^n T_a(e_{ij}) \otimes e_{ij} \\
&=& \frac{1}{n} \, \sum_{i,j=1}^n (\id_n \otimes \tau_N)\big(a(e_{ij} \otimes 1_N)a^*\big)\otimes e_{ij} \,\,=\,\, \frac{1}{n} \, \sum_{i,j=1}^n \sum_{k,\ell=1}^n \tau_N(a_{ik}a_{j\ell}) \, e_{ij} \otimes e_{k\ell}\,.
\end{eqnarray*}
It follows from  Theorem \ref{thm:HW} (2) that
$$F(T_a) = \Tr(\widehat{T_a} \, p^+) \, W_n^+ + \Tr(\widehat{T_a} \, p^-) \, W_n^-,$$
where $p^\pm = p_n^\pm$ are as in \eqref{eq:PS}. Recall that $p^\pm = \frac12(1 \pm s)$, where $s=s_n$ is defined in \eqref{eq:S}. Hence
$(e_{ij} \otimes e_{k\ell}) \, p^\pm = \big(e_{ij} \otimes e_{k\ell} \pm e_{i\ell} \otimes e_{kj}\big)/2.$
Thus
\begin{eqnarray*}
\Tr(\widehat{T_a} \, p^\pm) &=& \frac{1}{2n} \,  \sum_{i,j=1}^n \sum_{k,\ell=1}^n \tau_N(a_{ik}a_{j\ell}) \, \Tr(e_{ij} \otimes e_{k\ell} \pm e_{i\ell} \otimes e_{kj}) \\
&=& \frac{1}{2n} \,  \sum_{i,k=1}^n \tau_N(a_{ik}a_{ik}^*) \, \pm \,
\frac{1}{2n} \,  \sum_{i,k=1}^n \tau_N(a_{ik}a_{ki}^*) \\
&=& \frac{1}{2} \, (\tau_n \otimes \tau_N)\big(aa^* \pm a(a^*)^t\big) \; = \;
\frac14 \, \|a \pm a^t\|_2^2.
\end{eqnarray*}
This proves the claim.
\end{proof}

\begin{theorem} \label{thm12} \mbox{} The following hold: \vspace{.1cm}
\begin{enumerate}
\item $\displaystyle{\frac{2}{27} \, W_3^+ + \frac{25}{27} \, W_3^-}$ has an exact factorization through $M_3(\C) \otimes M_3(\C)$. \vspace{.1cm}
\item $\displaystyle{d_{\text{cb}}\big(W_3^-, {\mathcal F}{\mathcal M}(M_3(\C)) \big) = \frac{4}{27}}$.
\end{enumerate}
\end{theorem}

\begin{proof} Let $s = s_3\in M_3(\C) \otimes M_3(\C)$ be the flip symmetry defined in \eqref{eq:S} and let $q = q_3 \in M_3(\C) \otimes M_3(\C)$ be the projection defined in \eqref{eq:q}. Since $sq=q$ it follows that $u:= s-2q$ is unitary in $M_3(\C) \otimes M_3(\C)$. We claim that
\begin{equation} \label{eq:962}
\frac{2}{27} \, W_3^+ + \frac{25}{27} \, W_3^- = (\id_3 \otimes \tau_3)\big(u(x \otimes 1_3)u^*\big), \qquad x \in M_3(\C),
\end{equation}
from which $(1)$ will follow. Set $a_{ij} = e_{ij} - e_{ji}$ and $b_{ij} = e_{ij} + e_{ji}$\,, $1 \le i < j \le 3$. Recall from \eqref{eq:HW2} that
$$
W_3^+(x)  = \frac14 \Big(\; \sum_{1 \le i < j \le 3} b_{ij}xb_{ij}^* + 2\: \sum_{i=1}^3 e_{ii}xe_{ii}^*\Big), \qquad W_3^- (x) = \frac12 \, \sum_{1 \le i < j \le 3} a_{ij}xa_{ij}^*\,,$$
for all $x \in M_3(\C)$. Moreover,
$$u = \frac13 \, \sum_{i=1}^3 e_{ii} \otimes e_{ii} + \frac16 \, \sum_{1 \le i < j \le 3} b_{ij} \otimes b_{ij} - \frac56 \, \sum_{1 \le i < j \le 3} a_{ij} \otimes a_{ij}.$$
Since $\{e_{11}, e_{22}, e_{33}, b_{12}, b_{13}, b_{23}, a_{12}, a_{13}, a_{23}\}$ is an orthonormal set in $M_3(\C)$ with respect to the inner product arising from $\tau_3$, and since
$\|e_{ii}\|_2^2 = 1/3$ and $\|a_{ij}\|_2^2 = \|b_{ij}\|_2^2 = 2/3,$
then for all $x \in M_3(\C)$
\begin{eqnarray*}
(\id_3 \otimes \tau_3)\big(u(x \otimes 1_3)u^*\big) &=&
\frac13 \Big(\frac13\Big)^2 \sum_{i=1}^3 e_{ii}xe_{ii}^* \, + \,
\frac23 \Big(\frac16\Big)^2 \sum_{1 \le i < j \le 3} b_{ij} x b_{ij}^* \, + \,
\frac23 \Big(\frac56\Big)^2 \sum_{1 \le i < j \le 3} a_{ij} x a_{ij}^* \\
&=& \frac{2}{27} \, W_3^+(x) + \frac{25}{27} \, W_3^-(x).
\end{eqnarray*}
This proves \eqref{eq:962}.

(2). It follows from (1) and Lemma \ref{dist=2} that
\begin{eqnarray*}
d_{\text{cb}}\big(W_3^-, {\mathcal F}{\mathcal M}(M_3(\C)) \big)  &\le & \left\| W_3^- - \left(\frac{2}{27} \, W_3^+ + \frac{25}{27} \, W_3^- \right) \right\|_\cb \\
&\le & \frac{2}{27} \, \|W_3^- - W_3^+\|_\cb \; = \: \frac{4}{27}.
\end{eqnarray*}
To prove the reverse inequality, let $T \in {\mathcal F}{\mathcal M}(M_3(\C))$. Then
$T(x) = (\id_3 \otimes \tau_N)\big(u(x \otimes 1_N)u^*\big)$, $x \in M_3(\C),$
for some finite von Neumann algebra $N$ with faithful normal trace $\tau_N$ and some unitary operator $u \in M_3(N)$.   By Lemma \ref{lemma16},
$$F(T) = \frac14 \, \|u+u^t\|_2^2 \; W_n^+ + \frac14 \, \|u-u^t\|_2^2 \; W_n^- = \lambda W_n^+ + (1-\lambda) W_n^-,$$
where $\lambda = (1/4) \, \|u+u^t\|_2^2$. (By the parallelogram identity, $(1/4) \|u+u^t\|_2^2 + (1/4) \|u-u^t\|_2^2 = (1/2)\|u\|_2^2 + (1/2) \|u^t\|_2^2= 1$.)
Recall from Lemma \ref{lemma14} that $\|u-u^t\|_2^2 \le 25/27$. Hence $\lambda \ge 2/27$. Since the twirl map $F$ is a  complete contraction and  $F(W_3^-) = W_3^-$, it follows that
\begin{eqnarray*}
\|W_3^- - T\|_\cb & \ge & \|W_3^- - F(T) \|_\cb \\
&=& \|W_3^- - \big( \lambda W_n^+ + (1-\lambda) W_n^-\big)\|_\cb
\;=\; \lambda \, \|W_3^- - W_3^+\|_\cb \; = \; 2 \lambda \; \ge \; {4}/{27}\,,
\end{eqnarray*}
wherein we have used Lemma \ref{dist=2}.
\end{proof}

The following corollary follows now immediately by convexity of the set of factorizable maps:

\begin{cor}\label{hwfactorizable} Let $0 \le \lambda \le 1$. Then
$$\lambda \, W_3^+ + (1-\lambda) \, W_3^- \in {\mathcal F}{\mathcal M}(M_3(\C))$$
if and only if $\lambda \ge 2/27$.
\end{cor}

\section{The case of three tensors $T_\lambda\otimes T_\lambda\otimes T_\lambda$}

For $\lambda \in [0, 1]$, set $T_\lambda\colon = \lambda W_3^{+} + (1-\lambda) W_3^{-}$. As we have seen (cf. Corollary \ref{hwmixtureunit}),
$T_\lambda\in \text{conv}(\text{Aut}(M_3(\mathbb{C}))$ if and only if $\lambda\in [1/3, 1]$\,,
and, respectively (cf. Corollary \ref{hwfactorizable}),
$T_\lambda\in {\mathcal F}{\mathcal M}(M_3(\C))$ if and only if $\lambda\in [2/{27}, 1]$\,.
Further, Mendl and Wolf proved in \cite{MWo} that for some $0< \lambda_0< {1/3}$, one has
\begin{equation}\label{twotensorsmixtureunit}
T_\lambda\otimes T_\lambda \in \text{conv}(\text{Aut}(M_9(\mathbb{C}))\,, \,\,\, \text{for all} \,\,\, \lambda\in [\lambda_0, 1]\,.
\end{equation}
The value $\lambda_0$ is not stated explicitly in \cite{MWo}, but following the details of their proof one can show (see Remark \ref{rem65} $(ii)$ below) that (\ref{twotensorsmixtureunit}) holds for
\[ \lambda_0=(\sqrt{2}-1)\left(1-\frac{1}{\sqrt{3}}\right)\thickapprox 0.17507\,. \]
It follows that for all $\lambda\in [\lambda_0, 1]$ and for all \underline{even} integers $k\geq 2$, one has
$T_\lambda^{\otimes k}\in \text{conv}(\text{Aut}(M_{3^k}(\mathbb{C})))$\,.
However, the results in \cite{MWo} do not imply that, e.g., $T_\lambda\otimes T_\lambda\otimes T_\lambda \in \text{conv}(\text{Aut}(M_{27}(\mathbb{C}))$, for any $\lambda \in (0, 1/3)$.

Our main result in this section is the following:

\begin{theorem}\label{thm61}
For every $\lambda\in [1/4, 1]$ and every integer $k\geq 2$\,,
\[ T_\lambda^{\otimes k}\in \text{conv}(\text{Aut}(M_{3^k}(\mathbb{C})))\,. \]
\end{theorem}

Since $\text{Aut}(M_p(\mathbb{C}))\otimes \text{Aut}(M_q(\mathbb{C}))\subseteq \text{Aut}(M_p(\mathbb{C})\otimes M_q(\mathbb{C}))$, for all positive integers $p, q$, it is clearly sufficient to prove that for all $\lambda\in [1/4, 1]$,
\begin{equation}\label{eq61}
T_\lambda\otimes T_\lambda\in \text{conv}(\text{Aut}(M_{9}(\mathbb{C})))\,,
\end{equation}
and, respectively,
\begin{equation}\label{eq62}
T_\lambda\otimes T_\lambda\otimes T_\lambda\in \text{conv}(\text{Aut}(M_{27}(\mathbb{C})))\,.
\end{equation}
In view of Mendl and Wolf's result (cf. above comments), it suffices to prove $(\ref{eq62})$. Nonetheless, in the process of establishing $(\ref{eq62})$, we will also provide an elementary proof of $(\ref{eq61})$, based on ideas from \cite{MWo}.

Before proving Theorem \ref{thm61}, we establish some preliminary results. Let $F\in {\mathcal B}(M_3(\mathbb{C}))$ be the twirling map introduced in Definition \ref{def:twirlmap}, considered in the case $n=3$. Then $F\otimes F\in {\mathcal B}(M_3(\mathbb{C})\otimes M_3(\mathbb{C}))$ is given by
\begin{equation}\label{eq:twirltwice}
(F\otimes F)(T)=\int_{{\mathcal U}(3)\times {\mathcal U}(3)} \text{ad}(u\otimes v)T \text{ad}(u^t\otimes v^t) du dv\,, \quad T\in \text{UCPT}(3)\,.
\end{equation}
In particular, we have as in Proposition \ref{prop:twirl} that
\begin{equation}\label{eq3030}
(F\otimes F)(\mathrm{UCPT}(9))\subseteq (\mathrm{UCPT}(9))\,,
\end{equation}
respectively, that
\begin{equation}\label{eq3031}
(F\otimes F)(\text{conv}(\text{Aut}(M_{9}(\mathbb{C})))\subseteq \text{conv}(\text{Aut}(M_{9}(\mathbb{C}))\,.
\end{equation}
To simplify notation, we set in this section
\begin{equation}\label{eq3032}
W^{+}=W_3^{+}\,, \quad W^{-}=W_3^{-}\,,
\end{equation}
where $W_3^{+}$ and $W_3^{-}$ are the Holevo-Werner channels in dimension $n=3$ defined in $(\ref{eq:HW})$. Moreover, let $S$ and $A$, respectively, denote the symmetrization (resp., anti-symmetrization) map on $M_3(\mathbb{C})$, that is,
\begin{eqnarray}
S(a)&=& (a+a^t)/2\,, \quad a\in M_3(\mathbb{C})\label{eq3033}\,,\\
A(a)&=& (a-a^t)/2\,, \quad a\in M_3(\mathbb{C})\,.\label{eq3034}
\end{eqnarray}

\begin{lemma}\label{lem62}
Let $u\in {\mathcal U}(M_3(\mathbb{C})\otimes M_3(\mathbb{C}))$. Then
\begin{eqnarray}\label{eq3035}
(F\otimes F)(\text{ad}(u))&=&\|(S\otimes S)(u)\|_2^2 \,W^{+}\otimes W^{+} + \|(S\otimes A)(u)\|_2^2 \,W^{+}\otimes W^{-}\label{eq3035}\\&& \,\,+ \,\|(A\otimes S)(u)\|_2^2 \,W^{-}\otimes W^{+}+\|(A\otimes A)(u)\|_2^2 \,W^{-}\otimes W^{-}\nonumber
\end{eqnarray}
\end{lemma}

\begin{proof}
Given $a, b\in M_3(\mathbb{C})$, define $T_a, T_{a, b}\in {\mathcal B}(M_3(\mathbb{C}))$ by
\[ T_{a, b}(x)=axb^*\,, \,\,\,\, T_a(x)=T_{a, a}(x)=axa^*\,, \quad x\in M_3(\mathbb{C})\,. \]
Then for $a\in M_3(\mathbb{C})$ we have by Theorem \ref{thm:HW} (3) that
\[ F(T_a)=\|S(a)\|_2^2 \,W^{+} + \|A(a)\|_2^2\,W^{-}\,. \]
Hence, by using the polarization identity $T_{a, b}=(1/4)(T_{a+b}-T_{a-b}+iT_{a+ib}-iT_{a-ib})$, it follows that
\[ F(T_{a, b})=\langle S(a), S(b)\rangle W^{+} + \langle A(a), A(b)\rangle W^{-}\,, \]
where $\langle a, b\rangle \colon =\tau_3(b^*a)$. Furthermore, for $a_1, a_2, b_1, b_2\in M_3(\mathbb{C})$, $T_{a_1\otimes a_2, b_1\otimes b_2}=T_{a_1, b_1}\otimes T_{a_2, b_2}$, and hence
\begin{eqnarray*}
(F\otimes F)(T_{a_1\otimes a_2, b_1\otimes b_2})&=& F(T_{a_1, b_1})\otimes F(T_{a_2, b_2})\\
&=&(\langle S(a_1), S(b_1)\rangle W^{+} + \langle A(a_1), A(b_1)\rangle W^{-})\\&& \,\,\otimes \,\,(\langle S(a_2), S(b_2)\rangle W^{+} + \langle A(a_2), A(b_2)\rangle W^{-}).
\end{eqnarray*}
Since the map $(a, b)\mapsto T_{a, b}$ is linear in the first variable and conjugate-linear in the second, it follows that for all $a, b\in M_3(\mathbb{C})\otimes M_3(\mathbb{C})$,
\begin{eqnarray*}
(F\otimes F)(T_{a, b})&=& \langle (S\otimes S)(a), (S\otimes S)(b)\rangle \,W^{+}\otimes W^{+} \,+\, \langle (S\otimes S)(a), (S\otimes A)(b)\rangle \,W^{+}\otimes W^{-} \\
&&\,\,+ \, \langle (A\otimes S)(a), (A\otimes S)(b)\rangle \,W^{-}\otimes W^{+} \,+\, \langle (A\otimes A)(a), (A\otimes A)(b)\rangle \,W^{-}\otimes W^{-}\,.
\end{eqnarray*}
The conclusion follows now by taking $a=b=u$ in the above equation.
\end{proof}

The following lemma can be extracted from Mendl and Wolf's paper \cite{MWo}, cf. Remark \ref{rem65} $(i)$ below. We present here a more direct proof, based on ideas from \cite{MWo}.

\begin{lemma}\label{lem63} \mbox{}
Set $W_m= (W^{+}\otimes W^{-} + W^{-}\otimes W^{+})/2$. Then the operators
\[ Q_1=W^{+}\otimes W^{+}\,, \,\,\, Q_2={\frac{2}{27}} W^{+}\otimes W^{+} + {\frac{25}{27}} W^{-}\otimes W^{-}\,, \,\,\, Q_3={\frac{2}{3}} W_m + {\frac13} W^{-}\otimes W^{-} \]
are all contained in $\text{conv}(\text{Aut}(M_3(\mathbb{C})\otimes M_3(\mathbb{C})))$\,.
\end{lemma}

\begin{proof}
The statement about $Q_1$ is clear from Corollary \ref{cor:MWmixture} (1). Consider next the unitary $u$ from the proof of Theorem \ref{thm12}, namely,
\[ u = \frac13 \, \sum_{i=1}^3 e_{ii} \otimes e_{ii} + \frac16 \, \sum_{1 \le i < j \le 3} b_{ij} \otimes b_{ij} - \frac56 \, \sum_{1 \le i < j \le 3} a_{ij} \otimes a_{ij}, \]
where $ b_{i,j}=e_{ij}+e_{ji}$ and $a_{i, j}=e_{i,j}-e_{ji}$, $1\leq i< j\leq 3$. Then
\begin{equation*}
(S\otimes S)(u) \,=\, \frac13 \, \sum_{i=1}^3 e_{ii} \otimes e_{ii} + \frac16 \, \sum_{1 \le i < j \le 3} b_{ij} \otimes b_{ij}\,, \quad
(A\otimes A)(u) \,= \, - \frac56 \, \sum_{1 \le i < j \le 3} a_{ij} \otimes a_{ij}\,.
\end{equation*}
Moreover, $(S\otimes A)(u)=0=(A\otimes S)(u)$. Hence, by Lemma \ref{lem62},
\[ (F\otimes F)(\text{ad}(u))=\|(S\otimes S)(u)\|_2^2 \,W^{+}\otimes W^{+} + \|(A\otimes A)(u)\|_2^2 \,W^{-}\otimes W^{-}\,. \]
Using the orthogonality of the set of vectors $\{e_{11}, e_{22}, e_{33}, b_{12}, b_{13}, b_{23}, a_{12}, a_{13}, a_{23}\}$, together with the fact that $\|e_{ii}\|_2^2=1/3$, $\|a_{ij}\|_2^2=\|b_{ij}\|_2^2=2/3$, $1\leq i< j\leq 3$, we obtain that
\[ \|(S\otimes S)(u)\|_2^2=2/{27}\,, \quad \|(A\otimes A)(u)\|_2^2={25}/{27}\,. \]
Combined with $(\ref{eq3031})$, this shows that $Q_2=(F\otimes F)(\text{ad}(u))\in \text{conv}(\text{Aut}(M_3(\mathbb{C})\otimes M_3(\mathbb{C})))$\,.
Consider finally the matrix $v\in {M_3(\mathbb{C})\otimes M_3(\mathbb{C})}$ given by
\[ v=v_1+\omega v_2+ \bar{\omega} v_3\,, \]
where $\omega=({-1}/2) +i({\sqrt{3}}/2)$, $\bar{\omega}=({-1}/2) -i({\sqrt{3}}/2)$ and
\begin{equation*}
v_1=e_{12}\otimes e_{12} + e_{23}\otimes e_{23} + e_{31}\otimes e_{31}\,, \,\,\, v_2= e_{12}\otimes e_{21} + e_{23}\otimes e_{32} + e_{31}\otimes e_{13}\,,\end{equation*} \begin{equation*}v_3=e_{21}\otimes e_{12} + e_{32}\otimes e_{23} + e_{13}\otimes e_{31}\,. \end{equation*}
Note that $|\omega|=|\bar{\omega}|=1$. By the standard identification of $M_3(\mathbb{C})\otimes M_3(\mathbb{C})$ with $M_9(\mathbb{C})$, we have
$$v = \begin{pmatrix} 0 & 0 & 0 & 0 & 1 & 0 & 0 & 0 & 0 \\ 0 & 0 & 0 & \omega & 0 & 0 & 0 & 0 & 0 \\  0 & 0 & 0 & 0 & 0 & 0 & \bar{\omega} & 0 & 0 \\ 0 & \bar{\omega} & 0 & 0 & 0 & 0 & 0 & 0 & 0  \\ 0 & 0 & 0 & 0 & 0 & 0 & 0 & 0 & 1 \\ 0 & 0 & 0 & 0 & 0 & 0 & 0 & \omega & 0 \\ 0 & 0 & \omega & 0 & 0 & 0 & 0 & 0 & 0 \\ 0 & 0 & 0 & 0 & 0 & \bar{\omega} & 0 & 0 & 0 \\ 1 & 0 & 0 & 0 & 0 & 0 & 0 & 0 & 0 \end{pmatrix},$$
which shows that $v$ is unitary. Moreover, using $\omega + \bar{\omega}=-1$ and $\omega-\bar{\omega}=i \sqrt{3}$, we have the following
\begin{eqnarray*}
(S\otimes S)(v)&=&(1+ \omega + \bar{\omega})=0\,,\\ (A\otimes A)(v)&=&(1-\omega -\bar{\omega})(A\otimes A)(v_1)\,\,=\,\,{\frac12} \sum_{1\leq i< j\leq 3} a_{ij}\otimes a_{ij}\,, \\
(S\otimes A)(v)&=&(1-\omega-\bar{\omega})(S\otimes A)(v_1)\,\,=\,\,{\frac{1-i\sqrt{3}}{4}}\left(b_{12}\otimes a_{12} + b_{23}\otimes a_{23} - b_{13}\otimes a_{13}\right)\,,\\
(A\otimes S)(v)&=&{\frac{1+i\sqrt{3}}{4}}\left(a_{12}\otimes b_{12} + a_{23}\otimes b_{23} - a_{13}\otimes b_{13}\right)\,. \end{eqnarray*}
Hence, \[ \|(A\otimes A)(v)\|_2^2={\frac14}\sum_{i<j} \|a_{ij}\|_2^4=\frac13\,, \,\,\,\,\|(S\otimes A)(v)\|_2^2=\|(A\otimes S)(v)\|_2^2={\frac14}\sum_{i<j} \|a_{ij}\|_2^2\|b_{ij}\|_2^2=\frac13\,. \]
By Lemma \ref{lem62} and $(\ref{eq3031})$, we deduce that
\[ Q_3=(F\otimes F)(\text{ad}(v))\in \text{conv}(\text{Aut}(M_3(\mathbb{C})\otimes M_3(\mathbb{C})))\,, \]
which completes the proof.
\end{proof}

Next we will consider operators in ${\mathcal B}(M_3(\mathbb{C})\otimes M_3(\mathbb{C})\otimes M_3(\mathbb{C}))$. To simplify the notation, set
\begin{eqnarray*}
W^{+++}&=& W^{+}\otimes W^{+}\otimes W^{+}\,,\\W_m^{+}&=& {\frac13}(W^{+}\otimes W^{+}\otimes W^{-} + W^{+}\otimes W^{-}\otimes W^{+} + W^{-}\otimes W^{+}\otimes W^{+})\,,\\
W_m^{-}&=& {\frac13}(W^{+}\otimes W^{-}\otimes W^{-} + W^{-}\otimes W^{+}\otimes W^{-} + W^{-}\otimes W^{-}\otimes W^{+})\,,\\ W^{---}&=& W^{-}\otimes W^{-}\otimes W^{-}\,.
\end{eqnarray*}
Furthermore, let $\sigma\in {\mathcal B}(M_3(\mathbb{C})\otimes M_3(\mathbb{C})\otimes M_3(\mathbb{C}))$ denote the unique linear map for which
\begin{equation}\label{eq3037}
\sigma(T_1\otimes T_2\otimes T_3)=\frac{1}{|{\mathcal S}_3|} \sum_{\pi\in {\mathcal S}_3} T_{\pi(1)}\otimes T_{\pi(2)}\otimes T_{\pi(3)}\,, \quad T_1, T_2, T_3\in {\mathcal B}(M_3(\mathbb{C})\,,
\end{equation}
where ${\mathcal S}_3$ is the group of permutations of $\{1, 2, 3\}$ and $|{\mathcal S}_3|=6$. It is clear that $\sigma$ maps $\text{conv}(\text{Aut}(M_3(\mathbb{C})\otimes M_3(\mathbb{C}))\otimes M_3(\mathbb{C}))$ into itself.

\begin{lemma}\label{lem64}
The following four operators
\begin{eqnarray*}
R_1={\frac13} (W^{+++}+2W_m^{+})\,,\quad
R_2={\frac1{81}}(2 W^{+++} + 4 W_m^{+} + 25 W_m^{-} +50 W^{---})\,, \\
R_3= {\frac19}(2 W_m^{+} +5 W_m^{-} + 2 W^{---})\,,\quad
R_4= {\frac1{189}}(4 W^{+++} +168 W_m^{+} + 3 W_m^{-} +14 W^{---})
\end{eqnarray*}
are all contained in $\text{conv}(\text{Aut}(M_3(\mathbb{C})\otimes M_3(\mathbb{C}))\otimes M_3(\mathbb{C}))$\,.
\end{lemma}

\begin{proof}
Recall that $(1/3)W^{+} + (2/3) W^{-}\in \text{conv}(\text{Aut}(M_3(\mathbb{C})))$\,. Let $Q_1, Q_2, Q_3\in \text{conv}(\text{Aut}(M_3(\mathbb{C})\otimes M_3(\mathbb{C})))$ be as in Lemma \ref{lem63}. By a straightforward computation,
\[ R_i=\sigma\left(Q_i\otimes \left({\frac13} W^{+} + {\frac{2}{3}} W^{-}\right)\right)\,, \quad i=1, 2, 3\,. \]
Hence $R_1, R_2, R_3\in \text{conv}(\text{Aut}(M_3(\mathbb{C})\otimes M_3(\mathbb{C}))\otimes M_3(\mathbb{C}))$\,. To prove the same statement for $R_4$, we will show that \begin{equation}\label{eq3038}
R_4=\frac1{27} W_{27}^{+} + \frac{26}{27} W_{27}^{-}
\end{equation}
with respect to the standard identification of $M_3(\mathbb{C})\otimes M_3(\mathbb{C}))\otimes M_3(\mathbb{C})$ with $M_{27}(\mathbb{C})$\,. The desired conclusion will then follow by Corollary \ref{hwmixtureunit}.
We first show that
\begin{eqnarray}
W_{27}^{+}=\frac17 (4 W^{+++} + 3 W_m^{-})\,,\label{eq3040}\\
W_{27}^{-}=\frac1{13}(12 W_m^{+} + W^{---})\,,\label{eq3041}
\end{eqnarray}
from which $(\ref{eq3038})$ will follow immediately.
In order to prove $(\ref{eq3040})$ and $(\ref{eq3041})$, observe first that with respect to the standard identification of $M_3(\mathbb{C})\otimes M_3(\mathbb{C}))\otimes M_3(\mathbb{C})$ with $M_{27}(\mathbb{C})$, we have $S_{27}=S_3\otimes S_3\otimes S_3$\,,
where $S_k$ is the completely depolarizing channel in dimension $k\in \mathbb{N}$. Moreover, $t_{27}=t_3\otimes t_3\otimes t_3$\,,
where $t_k: x\mapsto x^t$ is the transposition map on $M_k(\mathbb{C})$. By $(\ref{eq:HW})$,
\[ W^{+}=(3 S_3 + t_3)/4\,, \quad W^{-}=(3 S_3 -t_3)/2\,. \]
Therefore, $S_3=(2 W^{+} + W^{-})/3$ and $t_3=2 W^{+} - W^{-}$\,.
Hence
\begin{equation}\label{eq3043}
S_{27}=\frac1{27}(2W^{+}+W^{-})^{\otimes 3}
= \frac1{27}(8 W^{+++} + 12 W_m^{+} + 6 W_m^{-} + W^{---})\,,
\end{equation}
and, respectively,
\begin{equation}\label{eq3044}
t_{27}=(2W^{+}+W^{-})^{\otimes 3}
= (8 W^{+++} - 12 W_m^{+} + 6 W_m^{-} - W^{---})\,.
\end{equation}
Since by $(\ref{eq:HW})$, \[ W_{27}^{+}=\frac1{28}(27 S_{27} + t_{27})\,, \quad W_{27}^{-}=\frac1{26}(27 S_{27} -t_{27})\,, \]
relations $(\ref{eq3043})$ and $(\ref{eq3044})$ imply $(\ref{eq3040})$ and $(\ref{eq3041})$, which prove $(\ref{eq3038})$, and complete the proof.
\end{proof}

\noindent
{\bf Proof of Theorem \ref{thm61}}: Consider the matrix
$$B = \begin{pmatrix} 0 & \frac{2}{27} & 0 \\ 0 & 0 & \frac{2}{3} \\  0 & \frac{25}{27} & \frac13 \end{pmatrix}\,. $$
Note that $\text{det}(B)\neq 0$. By Lemma \ref{lem63},
\[ (Q_1\,, Q_2\,, Q_3)=(W^{+}\otimes W^{+}\,, W_m\,, W^{-}\otimes W^{-}) B\,. \] Here $(x_1, x_2, \ldots, x_n)$ denotes the n-dimensional row vector with components $x_1, x_2, \ldots, x_n$.
Hence $(W^{+}\otimes W^{+}\,, W_m\,, W^{-}\otimes W^{-})=(Q_1\,, Q_2\,, Q_3) B^{-1}$\,. Note next that
\begin{eqnarray*}
{T_\lambda}\otimes {T_\lambda} &=& (\lambda W^{+} + (1-\lambda) W^{-})^{\otimes 2}\\
&=& \lambda^2 W^{+}\otimes W^{+} + 2 \lambda (1-\lambda) W_m +(1-\lambda)^2 W^{-}\otimes W^{-}\\
&=& (Q_1\,, Q_2\,, Q_3) B^{-1} {\begin{pmatrix} \lambda^2 \\ 2 \lambda (1-\lambda) \\ (1-\lambda)^2 \end{pmatrix}}\\
&=& p_1(\lambda) Q_1 + p_2(\lambda) Q_2 + p_3(\lambda) Q_3\,,
\end{eqnarray*}
where ${\begin{pmatrix} p_1(\lambda) \\ p_2(\lambda) \\ p_3(\lambda) \end{pmatrix}}=B^{-1}{\begin{pmatrix} \lambda^2 \\ 2 \lambda (1-\lambda) \\ (1-\lambda)^2 \end{pmatrix}}$\,.
Since $(1, 1, 1) B=(1, 1, 1)$, then also $(1, 1, 1) {B^{-1}}=(1, 1, 1)$\,. Hence
$p_1(\lambda) + p_2(\lambda) + p_3(\lambda)=\lambda ^2 +2\lambda (1-\lambda) +(1-\lambda)^2=1$\,.
It follows that if $p_i(\lambda)\geq 0$\,, $i=1, 2, 3$, then by Lemma \ref{lem63},
\[ {T_\lambda}\otimes {T_\lambda}\in {\text{conv}(\{Q_1, Q_2, Q_3\})}\subseteq \text{conv}(\text{Aut}(M_3(\mathbb{C})\otimes M_3(\mathbb{C})))\,. \]
Elementary computations in MAPLE or MATHEMATICA yield
\begin{eqnarray*}
p_1(\lambda)= {\frac1{25}}(21 \lambda^2 + 6 \lambda -2)\,,\,\,\,
p_2(\lambda)= {\frac{27}{25}}( 2 \lambda^2 -3 \lambda +1)\,,\,\,\,
p_3(\lambda)= 3 \lambda (1-\lambda)\,.
\end{eqnarray*}
The roots of $p_1$ are $\lambda_1^{+}={(-3 + \sqrt{51})/{21}}\approx 0.19721$ and $\lambda_1^{-}={(-3-\sqrt{51})/{21}}< 0$\,,
while the roots of $p_2$ are $1$ and $1/2$\,. Hence $p_i(\lambda)\geq 0$\, $i=1, 2, 3$ when
\begin{equation}\label{eq3056}
\lambda_1^{+}\leq \lambda \leq {1/2}\,.
\end{equation}
Thus ${T_\lambda}\otimes {T_\lambda}\in \text{conv}(\text{Aut}(M_3(\mathbb{C})\otimes M_3(\mathbb{C})))$, when $\lambda$ satisfies $(\ref{eq3056})$\,. Since, on the other hand, $T_\lambda\in \text{conv}(\text{Aut}(M_3(\mathbb{C})))$ when $\lambda\in [1/3, 1]$, we have altogether shown that
\begin{equation}\label{q3057}
{T_\lambda}\otimes {T_\lambda}\in \text{conv}(\text{Aut}(M_3(\mathbb{C})\otimes M_3(\mathbb{C})))\,, \quad \text{for} \,\, \lambda\in [\lambda_1^{+}, 1]\,.
\end{equation}
Since $\lambda_1^{+}< 1/4$, this implies $(\ref{eq61})$. We next prove $(\ref{eq62})$ in a similar way, this time by applying Lemma \ref{lem64}. Consider the matrix
$$ C=\begin{pmatrix} \frac13 & \frac{2}{81} & 0 & \frac{4}{189}\\ \frac{2}{3} & \frac{4}{81} & \frac{2}{9} & \frac{168}{189} \\  0 & \frac{25}{81} & \frac{5}{9} & \frac{3}{189} \\  0 & \frac{50}{81} & \frac{2}{9} & \frac{14}{189} \end{pmatrix}\,. $$
Then $\text{det}(C)\neq 0$, and by Lemma \ref{lem64},
\[ (R_1\,, R_2\,, R_3\,, R_4)=(W^{+++}\,, W_m^{+}\,, W_m^{-}\,, W^{---}) C\,. \]
Therefore $(W^{+++}\,, W_m^{+}\,, W_m^{-}\,, W^{---})=(R_1\,, R_2\,, R_3\,, R_4) C^{-1}$\,. Note next that
\begin{eqnarray*}
{T_\lambda}\otimes {T_\lambda} \otimes {T_\lambda}&=& (\lambda W^{+} + (1-\lambda) W^{-})^{\otimes 3}\\
&=& \lambda^3 W^{+++} + 3 \lambda^2 (1-\lambda) W_m^{+} +3\lambda (1-\lambda)^2  W_m^{-} + (1-\lambda)^3 W^{---}\,.
\end{eqnarray*}
By arguing further as in the proof of $(\ref{eq61})$ above, we have ${T_\lambda}\otimes {T_\lambda} \otimes T_\lambda=\sum_{i=1}^4 q_i(\lambda) R_i$\,,
where $q_1, q_2, q_3, q_4$ are the polynomials in $\lambda$ given by
\[ {\begin{pmatrix} q_1(\lambda) \\ q_2(\lambda) \\ q_3(\lambda) \\ q_4(\lambda) \end{pmatrix}}=C^{-1}{\begin{pmatrix} \lambda^3 \\ 3 \lambda^2 (1-\lambda) \\ 3 \lambda (1-\lambda)^2 \\ (1-\lambda)^3 \end{pmatrix}}\,. \]
Moreover, $q_1(\lambda) + q_2(\lambda) + q_3(\lambda) + q_4(\lambda)=1$\,. Hence, if $q_i(\lambda)\geq 0$\,, $i=1, 2, 3, 4$, then by Lemma \ref{lem64},
\[ {T_\lambda}\otimes {T_\lambda}\otimes {T_\lambda}\in {\text{conv}(\{R_1, R_2, R_3, R_4\})}\subseteq \text{conv}(\text{Aut}(M_3(\mathbb{C})\otimes M_3(\mathbb{C})\otimes M_3(\mathbb{C})))\,. \]
Explicit computations in MAPLE or MATHEMATICA yield
\begin{eqnarray*}
q_1(\lambda)= {\frac{15}{4}}\lambda^3 - \frac{123}{100} \lambda^2 + \frac{77}{100} \lambda - \frac{149}{900}&,& \,\,\,
q_2(\lambda)= -{\frac{27}{8}}\lambda^3 + \frac{1971}{200} \lambda^2 - \frac{1629}{200} \lambda + \frac{397}{200}\,,\\
q_3(\lambda)= {\frac{15}{2}}\lambda^3 - \frac{33}{20} \lambda^2 + 10 \lambda - \frac{10}{9}&,& \,\,\,
q_4(\lambda)= -{\frac{7}{24}}(3 \lambda-1)^3\,.
\end{eqnarray*}
The polynomial $q_1$ has only one real root, $\lambda_1\approx 0.23971$\,.
The polynomial $q_2$ has three distinct real roots: $\lambda_2^{(1)}\approx 0.45606$\,, $\lambda_2^{(2)}\approx 0.75435$ and $\lambda_2^{(3)}\approx 1.70959$\,. Respectively, the polynomial $q_3$ has also three distinct real roots:
$\lambda_3^{(1)}\approx 0.14241$\,, $\lambda_3^{(2)}\approx 0.89425$ and $\lambda_3^{(3)}\approx 1.16334$\,,
while $q_4$ has only one root $\lambda_4=1/3$, with multiplicity $3$. Taking into account the sign of the leading terms in $q_i(\lambda)$\,, $i=1, 2, 3, 4$, we deduce that $q_i(\lambda)\geq 0$\,, $i=1, 2, 3, 4$\,, whenever $\lambda\in [\lambda_1, 1/3]$\,.
It follows that for all $\lambda\in [\lambda_1, 1/3]$,
\begin{equation}\label{eq3059}
{T_\lambda}\otimes {T_\lambda}\otimes {T_\lambda}\in \text{conv}(\text{Aut}(M_3(\mathbb{C})\otimes M_3(\mathbb{C})\otimes M_3(\mathbb{C})))\,.
\end{equation}
Combining this with the fact that $T_\lambda\in \text{conv}(\text{Aut}(M_3(\mathbb{C})))$ when $\lambda\in [1/3, 1]$, we conclude that $(\ref{eq3059})$ holds for all $\lambda \in [\lambda_1, 1]$\,. This proves $(\ref{eq62})$, since $\lambda_1< {1/4}$, and completes the proof of Theorem \ref{thm61}.

\begin{rem}\label{rem65}
$(i)$ As mentioned at the beginning of this section, a different proof of Lemma \ref{lem63} can be extracted from Mendl and Wolf's paper \cite{MWo}. We will briefly explain the ideas of their proof using our terminology. Let $\sigma_2$ be the unique linear map on ${\mathcal B}(M_3(\mathbb{C})\otimes M_3(\mathbb{C}))$ for which
\[ \sigma_2(T_1\otimes T_2)=(T_1\otimes T_2 + T_2\otimes T_1)/2\,, \quad T_1, T_2\in {\mathcal B}(M_3(\mathbb{C}))\,, \]
and set $\beta\colon=\sigma_2\circ (F\otimes F)$\,. Then by $(\ref{eq3031})$,
\begin{equation}\label{eq4040}
\beta(\text{conv}(\text{Aut}(M_{9}(\mathbb{C})))\subseteq \text{conv}(\text{Aut}(M_{9}(\mathbb{C}))\,,
\end{equation}
and by Lemma \ref{lem62}, we also have
\begin{equation}\label{eq4041}
\beta(\text{conv}(\text{Aut}(M_{9}(\mathbb{C})))\subseteq \text{conv}(\{W^{+}\otimes W^{+}\,, W_m\,, W^{-}\otimes W^{-}\})\,,
\end{equation}
Consider next the unique affine map $\alpha\colon \text{conv}(\{W^{+}\otimes W^{+}\,, W_m\,, W^{-}\otimes W^{-}\})\rightarrow {\mathbb{R}}^2$ for which
\[ \alpha(W^{+}\otimes W^{+})=(1, 1)\,, \,\,\, \alpha(W_m)=(-1, 0)\,, \,\,\, \alpha(W^{-}\otimes W^{-})=(-1, 1)\,. \]
In \cite[Sections V.A and VII.C]{MWo} it is proved that the set
\begin{equation}\label{eq4045}
{\mathcal A}=(\alpha\circ \beta)(\text{Aut}(M_{9}(\mathbb{C})))
\end{equation}
contains the blue area (here denoted by ${\mathcal A}_0$) (see \cite[Fig. 3, p.17]{MWo}). Note that ${\mathcal A}_0$ is the convex hull of two points, namely, $(1, 1)$ and $(1/9, -{7}/{9})$, together with the path $\Gamma=\{\gamma(t): t\in [0, 1]\}$ in ${\mathbb{R}}^2$ given by
\begin{equation}\label{eq4046}
\gamma(t)={\frac1{9}}\left(-{\frac{8}{3}}(t+1)^2 +3\,, 16t^2-7\right)\,, \quad t\in [0, 1]\,,
\end{equation}
(cf. \,\cite[formulas (23), (36) and (37)]{MWo}). The path $\Gamma$ is obtained by an explicit construction of unitaries
\[ u(\theta)\in {\mathcal U}(M_3(\mathbb{C})\otimes M_3(\mathbb{C}))\,, \quad \theta \in [0, {\pi}/2]\,, \]
for which $(\alpha\circ \beta) (u(\theta))=\gamma(\cos \theta)$\,. We now have \[ \alpha(W^{+}\otimes W^{+})=(1, 1)\in {\mathcal A}_0\,, \]
and by letting $t=1$ and $t=1/2$, respectively, in $(\ref{eq4046})$, we deduce that
\begin{eqnarray*}
\alpha\left(\frac{2}{27} W^{+}\otimes W^{+} +\frac{25}{27} W^{-}\otimes W^{-}\right)=\left(-\frac{23}{27}, 1\right)\in \Gamma\subseteq {\mathcal A}_0\end{eqnarray*}
and, respectively,
\begin{eqnarray*}
\alpha\left(\frac{2}{3} W_m +\frac{1}{3} W^{-}\otimes W^{-}\right)= \left(-\frac13, -\frac13\right)\in \Gamma\subseteq {\mathcal A}_0\,.
\end{eqnarray*}
Since $\alpha$ is one-to-one, it follows that $Q_1$\,, $Q_2$ and $Q_3$ from Lemma \ref{lem63} are all contained in ${\alpha^{-1}}({\mathcal A})=
\beta(\text{conv}( \text{Aut}(M_{9}(\mathbb{C}))))\subseteq \text{conv}(\text{Aut}(M_{9}(\mathbb{C})))$\,, as claimed. \\
$(ii)$ Let $\alpha$, $\beta$, ${\mathcal A}$ and ${\mathcal A}_0$ be as defined above. Mendl and Wolf's proof of the fact that there exists $\lambda_0 \in (0, 1/3)$ such that $T_\lambda \otimes T_\lambda \in \text{conv}( \text{Aut}(M_{9}(\mathbb{C})))$\,, for all $\lambda\in [\lambda_0, 1]$,
is obtained by considering the path $\Lambda$ in $\mathbb{R}^2$ given by
\[ \Lambda=\{\alpha(T_\lambda\otimes T_\lambda): \lambda\in [0, 1]\}=\{(2\lambda-1, (2\lambda-1)^2): \lambda\in [0, 1]\}\,, \]
which is the orange-colored parabola in Fig. 3 of \cite{MWo}. The two paths $\Gamma$ and $\Lambda$ intersect in precisely one point, called $\rho_T$, whose first coordinate is equal to $-1/3-\varepsilon$, where $\varepsilon=(2/3)(4-3\sqrt{2}-\sqrt{3}+\sqrt{6})$, according to \cite[Sect. V.A]{MWo}. Hence $\rho_T=\alpha(T_{
\lambda_0}\otimes T_{\lambda_0})$, where $\lambda_0$ is determined by
\[ 2{\lambda_0}-1=-{1/3} -\varepsilon\,. \] Thus $\lambda_0=-{1/3}-{\varepsilon}/2=(\sqrt{2}-1)(1-1/{\sqrt{3}})\approx 0.17507$. By \cite[Fig. 3]{MWo}, $\alpha(T_\lambda\otimes T_{\lambda})\in {\mathcal A}_0$, for all $\lambda\in [\lambda_0, 1]$, and hence
\[ T_\lambda\otimes T_\lambda \in \beta(\text{conv}( \text{Aut}(M_{9}(\mathbb{C}))))\subseteq \text{conv}( \text{Aut}(M_{9}(\mathbb{C}))), \]
for all $\lambda\in [\lambda_0, 1]$\,.
\end{rem}

\thanks{}

\end{document}